\documentclass[10pt,final]{siamltex}
\usepackage{amsmath}
\usepackage{bm}
\usepackage{amssymb,version}
\usepackage{cases}
\usepackage{color}
\usepackage{verbatim}
\usepackage{multirow}
\usepackage{graphicx}
\usepackage{subfigure}
\usepackage{graphics}
\usepackage{epsfig}
\usepackage{enumitem}
\newtheorem{remark}{Remark}[section]

\usepackage{silence}
\usepackage{hyperref}
\allowdisplaybreaks
\begin{document}
\graphicspath{{figures/},}
    \title{Stability and error analysis of fully discrete original energy-dissipative and length-preserving scheme for the Landau-Lifshitz-Gilbert equation
\thanks{This work is supported by the National Natural Science Foundation of China (Grant Nos. 12271302, 12131014, 12271240, 12426312) and Shandong Provincial Natural Science Foundation for Outstanding Youth Scholar (Grant No. ZR2024JQ030)}, and NSF (Grant No. DMS-2309548)}
 \author{Binghong Li
        \thanks{School of Mathematics, Shandong University, Jinan, Shandong, 250100, P.R. China. Email: binghongsdu@163.com}.
        \and Xiaoli Li
        \thanks{Corresponding author. School of Mathematics and State Key Laboratory of Cryptography and Digital Economy Security, Shandong University, Jinan, Shandong, 250100, P.R. China. Email: xiaolimath@sdu.edu.cn}.
        \and Cheng Wang
        \thanks{Mathematics Department, University of Massachusetts, North Dartmouth, MA 02747 USA. Email: cwang1@umassd.edu}.
        \and Jiang Yang\thanks{Department of Mathematics, SUSTech International Center for Mathematics \& National Center for Applied Mathematics Shenzhen (NCAMS), Southern University of Science and Technology,
			Shenzhen, China. Email:yangj7@sustech.edu.cn.}
}

\maketitle
\begin{abstract} 
The Landau-Lifshitz-Gilbert (LLG) equation, regarded as a gradient flow with manifold constraint, is the fundamental model describing magnetization dynamics in ferromagnetic materials. This equation exhibits highly nonlinear behavior and involves a non-convex manifold constraint \(|\mathbf{m}| = 1\), along with energy dissipation property. 
It is well known that the normalized tangent plane method is able to simultaneously achieve the non-convex manifold constraint and original energy dissipation. However, the associated computational cost of this numerical approach is exceedingly high. By contrast, the projection method is more straightforward to implement, while it often compromises the inherent energy dissipative property of the continuous model, and the error analysis turns out to be even more challenging. In this work, we first construct a linear and fully discrete finite difference numerical scheme, based on the projection method for the LLG equation, which is capable of simultaneously preserving the non-convex manifold constraint \(|\mathbf{m}| = 1\) and an unconditional original energy dissipation.
In the error analysis, the classical theoretical technique becomes ineffective, due to the presence of the nonlinear Laplacian term, which in turn poses a significant challenge. To overcome this subtle difficulty, we carefully rewrite the numerical method in an equivalent weak form, in which a point-wise length preserving feature of the numerical solution plays an essential role. Based on such a reformulation, a nonlinear Laplacian estimate is avoided, and the rest nonlinear error bounds could be derived with the help of discrete Sobolev interpolation, as well as a Law-of-Cosine style estimate of the numerical errors at the renormalization stage. 
As a result of these estimates in the reformulated weak form, an optimal convergence rate could be theoretically established. 
In our knowledge, this numerical method is the first linear algorithm that preserves the following combined theoretical properties: (i) point-wise length preservation, (ii) unconditional original energy dissipation, (iii) a theoretical justification of convergence analysis and optimal rate error estimate. Some numerical experiments are presented to verify the theoretical findings and illustrate the robustness and effectiveness of the proposed method.
 \end{abstract}

  
 \section{Introduction}
 The magnetization structures in the form of topological solitons have become an important research topic, following the fabrication of various ferromagnetic materials. Static solutions of the Landau-Lifshitz type equation play a crucial role in the study of magnetization structures in various ferromagnetic materials \cite{jia2025electrically, jiang2015blowing, romming2013writing, nagaosa2013topological}. In this work we consider the Landau–Lifshitz–Gilbert (LLG) equation by taking the form of a time-dependent nonlinear equation as follows \cite{cai2022second, guo1993landau, landau1992theory}
  \begin{align}
   \setlength{\abovedisplayskip}{4pt} 
     & 
    \textbf{m}_t = -\beta \textbf{m} \times \mathbf{H_{eff}}
- \gamma  \textbf{m} \times ( \textbf{m} \times \mathbf{H_{eff}} ), \quad  \ {\rm in} \ \Omega \times J, 
  \label{e_original model1} 
\\
  & \mbox{with} \quad  
  \textbf{m}(\textbf{x},0 ) = \textbf{m}_0(\textbf{x}) , \text{ with } |\textbf{m}_0(\textbf{x}) |=1, \ {\rm in}\ \Omega, 
  \setlength{\belowdisplayskip}{4pt} 
  \label{e_initial condition}
\end{align} 
subject to either homogeneous Neumann or periodic boundary conditions, and $\mathbf{H_{eff}}$ is the effective field. In the above system, $\textbf{m}=(m_1,m_2,m_3)^T$ describes the magnetization in continuum ferromagnets, $\Omega$ is an open bounded domain in $\mathbb{R}^d$ with $d \in \{ 1,2,3\}$, $J$ denotes $(0, T]$ for some $T >0$, $\gamma>0$ is the Gilbert damping parameter and $\beta$ is an exchange parameter. When $\beta \neq 0$, it is often referred to as the Landau-Lifshitz-Gilbert equation. The solution of \eqref{e_original model} preserves a point-wise magnitude, i.e.
\begin{equation}\label{e_length preserving}
\setlength{\abovedisplayskip}{4pt} 
\frac{d}{dt} | \textbf{m}(\textbf{x},t)  |^2=0,
\setlength{\belowdisplayskip}{4pt} 
\end{equation}  
which, together with  \eqref{e_initial condition}, implies that $ | \textbf{m}(\textbf{x},t )|=1, \,\forall (x,\,t)$. It also satisfies the following energy dissipation law:
   \begin{align} 
   \setlength{\abovedisplayskip}{4pt} 
    & 
\frac{d E}{dt}  = - \gamma \| \textbf{m} \times \mathbf{H_{eff}} \|^2,\quad \mathbf{H_{eff}} = -\frac{\delta E}{\delta \textbf{m}}, \label{e_energy dissipative law}  
\\
  & \mbox{where} \quad  
	E(\textbf{m})=   \int_{\Omega} ( \frac{1}{2}|\nabla \textbf{m}|^2+\mathcal{D} (\nabla \times \textbf{m}) \cdot \textbf{m}  + K_u\Phi(\mathbf{m})
	-\mu_0  \textbf{H}_{\mathrm{ex}}\cdot \textbf{m} )d\textbf{x}. 
   \setlength{\belowdisplayskip}{4pt} 
	\label{e_general energy} 
\end{align} 
Here the first term denotes the exchange energy, which promotes the parallel alignment of neighboring magnetizations and $\mathcal{D}$ represents the Dzyaloshinskii-Moriya interaction (DMI) strength, which favors a perpendicular arrangement. The term $\Phi(\mathbf{m})$ corresponds to the anisotropy energy, and $K_u$ denotes the coefficient of the uniaxial anisotropy. The final term represents the Zeeman energy arising from the external magnetic field  $\textbf{H}_{\mathrm{ex}}= (0,0,1)$, with $\mu_0$ denoting the strength of an external magnetic field. By neglecting the effects of anisotropy, DMI energy and external fields, we consider the effective magnetic field as $\mathbf{H_{eff}} = \Delta \mathbf{m}$. In turn, system \eqref{e_original model1} is reduced to
\begin{equation}\label{e_original model}
 \setlength{\abovedisplayskip}{4pt} 
  \textbf{m}_t = -\beta \textbf{m} \times \Delta \textbf{m}
- \gamma  \textbf{m} \times ( \textbf{m} \times \Delta \textbf{m} ). 
\setlength{\belowdisplayskip}{4pt} 
\end{equation}

As be well known, the LLG equation \eqref{e_original model} admits several equivalent formulations. Among these choices, the three forms that have garnered the most significant attention in the numerical analysis over the past few decades are listed as follows.
   \begin{align}
    \setlength{\abovedisplayskip}{4pt} 
   &\textbf{(i)}\,\displaystyle 
         \textbf{m}_t - \gamma\Delta \textbf{m} = -\beta \textbf{m} \times \Delta \textbf{m}\, \textcolor{black}{ + }\, \gamma |\nabla\textbf{m}|^2\textbf{m},\label{eq1.6}\\
   &\textbf{(ii)}\,\displaystyle
        \textbf{m}_t + \lambda \textbf{m} \times \textbf{m}_t 
        = (1 + \lambda^2)\textbf{m}\times \Delta \textbf{m},\label{eq1.7} \\
     &\textbf{(iii)}\,\displaystyle
        \lambda \textbf{m}_t - \textbf{m} \times \textbf{m}_t 
        = (1 + \lambda^2)(\Delta\textbf{m} + |\nabla \textbf{m}|^2\textbf{m}),\label{eq1.8}
     \setlength{\belowdisplayskip}{4pt} 
    \end{align} 
where $\lambda > 0$ is the dimensionless Gilbert damping constant \cite{akrivis2021higher}, which can be expressed in terms of $\beta, \,\gamma$. In terms of the numerical solution of the LLG equation, preserving the intrinsic energy-dissipative and length-preserving properties throughout the evolution process remains a fundamental challenge. Ensuring these properties is of significant importance for maintaining the numerical accuracy and stability. Currently, three primary strategies have been employed: i) The penalty method \cite{badia2011finite,liu2000approximation,prohl2001computational} relaxes the constraint \(|\textbf{m}| = 1\) by incorporating a penalty term into the original equation. While this approach simplifies the implementation, it generally fails to strictly enforce the constraint condition. ii) The projection method \cite{cohen1989relaxation, kim2017mimetic, li2005numerical, suess2002time, yang1998dynamical} preserves the constraint $\displaystyle \textbf{m}^n = \mathbf{\widetilde{m}}^n/\,|\mathbf{\widetilde{m}}^n|$ by projecting the intermediate solution onto the unit sphere. Although straightforward to implement, this approach can compromise the original energy dissipation property inherent to the PDE system \eqref{e_original model}. iii) The Lagrange multiplier method \cite{cheng2023length} introduces an auxiliary variable to maintain the constraint condition, which requires the solution of a nonlinear algebraic equation and poses great challenges in terms of error analysis. Given its simplicity and efficiency, we prioritize the projection method approach in our work. Subsequently, we focus on addressing the issue of ensuring original energy dissipation in the projection method.

In fact, there have been a series of works related to projection method. E and Wang established a first-order error estimate for the projection scheme in \cite{weinan2001numerical}.  An et al. \cite{an2021optimal} presented an error estimate for first- and second-order semi-implicit projection finite difference methods, and their analysis relies on inverse inequalities and requires $h^2\le \Delta t\le h^{1+\epsilon_0}$. Such a constraint has been improved in a more recent work~\cite{an2022analysis}, to the scale of $\Delta t = O(\epsilon_0 h)$ with some small $\epsilon_0$ for a first-order finite element projection scheme. On the other hand,  Gui et al. \cite{gui2022convergence} derived an optimal-order error estimate for a linearly implicit, lumped mass FEM for the heat flow of harmonic maps on rectangular mesh under the condition $ \Delta t \geq  \kappa h^{r+1} $ for some $ r >1$, where $\kappa$ is any positive constant. Cai et al. \cite{cai2022second, Cai2023a} constructed the second order accurate numerical scheme for the Landau-Lifshitz-Gilbert equation with large damping parameters. The aforementioned methods fail to satisfy the energy dissipation law, thereby preventing the achievement of the desired stability results, which are crucial for rigorous error analysis. Moreover, in practical numerical schemes for the LLG equation, preserving the original energy dissipativity and stability property is of paramount physical significance. For instance, in the process of solving the skyrmion relaxation algorithm, maintaining the intrinsic energy dissipation property of the numerical method applied to \eqref{e_original model} is essential for ensuring convergence to a local energy minimum \cite{komineas2015skyrmion}.

To solve this issue, Li et al. \cite{li2026class} developed a class of higher-order implicit-explicit schemes using the generalized scalar auxiliary variable approach for the Landau-Lifshitz equation. The constructed schemes satisfy a modified energy dissipation law, while the original energy stability and dissipative property have not been theoretically justified. In addition, the non-orthogonal projection (or renormalization) method was proposed for \eqref{eq1.7} by Alouges and Jaisson \cite{alouges2006convergence}. At each time step, this method requires constructing a new finite element space that is pointwise orthogonal to the finite element solution obtained at the previous step. The convergence of this procedure has been rigorously established in \cite{alouges2006convergence}. Additionally, Alouges devised a discrete non-orthogonal projection scheme for equation \eqref{eq1.8} and demonstrated that this scheme preserves the energy dissipation and stability properties inherent to the continuous system \cite{alouges2008new}. Subsequently, An et al. \cite{an2025optimal} conducted a theoretical analysis of the normalized tangent plane method. While the renormalization approach provides a theoretical derivation of the original energy dissipation, it necessitates constructing a new finite element space at each time step, which incurs significant computational costs \cite{alouges2008new, an2025optimal}.

In fact, the energy dissipative structure on the unit sphere for the LLG equation can also be analyzed from the perspective of the gradient flow of vector fields with manifold constraint: 
\begin{equation}
\setlength{\abovedisplayskip}{4pt} 
    \begin{array}{l}
    \displaystyle
        \textbf{m}_t = P(\mathbf{m}) \,\mathbf{H_{eff}},   \quad 
        P(\mathbf{m}) = \gamma( I - \frac{\mathbf{m}\mathbf{m}^{T}}{|\mathbf{m}|^2}) - \beta \mathbf{m}\times\cdot,\\
         |\mathbf{m}(\mathbf{x},0)| = 1.
    \end{array}
\setlength{\belowdisplayskip}{4pt} 
\end{equation} 
From this perspective, Du et al. \cite{du2025semi} proposed the semi-implicit projection (SIP) method that preserves both the manifold constraint and the original energy dissipation:
\begin{equation}\label{semidiscrete}
  \setlength{\abovedisplayskip}{4pt} 
\left\{
    \begin{array}{l}
    \displaystyle
        \frac{\widetilde{\mathbf{m}}^{n+1} - \mathbf{m}^{n}}{\Delta t} = P(\mathbf{m}^n) \,\mathbf{H_{eff}}^{n+1},   \\
     \displaystyle    
       \mathbf{m}^{n+1} = \frac{\widetilde{\mathbf{m}}^{n+1}}{|\widetilde{\mathbf{m}}^{n+1}|}.
    \end{array} \right. 
    \setlength{\belowdisplayskip}{4pt} 
\end{equation} 
However, an extension of this approach to the fully discrete case and a rigorous error analysis have remained a theoretically significant challenge. 

The main purpose of this paper is to present and analyze the fully discrete finite difference scheme based on the semi-discrete approach \eqref{semidiscrete}. Our main contributions are outlined as follows. 
\begin{itemize}
    \item We first construct linear and fully discrete finite difference scheme based on the projection method for the LLG equation. The proposed scheme is straightforward to implement and capable of simultaneously enforcing the non-convex constraint \( |\mathbf{m}| = 1 \), ensuring an unconditional original energy dissipation and maintaining numerical stability in the discrete $H^1$ norm.
    \item Based on an equivalent weak reformulation of the original numerical algorithm, with the discrete summation by parts to avoid the highly complicated terms, we have rigorously derived an optimal error estimate for the proposed scheme in the discrete  \(L^\infty(0,\,T;\,L^2(\Omega))\) and \(L^2(0,\,T;H^1(\Omega))\)  norms.  These results hold under mild grid ratio conditions: in two-dimensional (2D) cases with \(\textcolor{black}{h^2 \lesssim \Delta t \lesssim h^{\epsilon_0}}\) and in three-dimensional (3D) cases with \(h^2 \lesssim \Delta t \lesssim h^{1+\epsilon_0}\) and \(h \leq h_0\), where $0<\epsilon_0<1$ and $h_0$ is a positive constant.  
\end{itemize}

Specifically, the error analysis for the fully discrete finite difference scheme 
is highly nontrivial, due to the complicated nonlinear structures. In the derivation of error estimates 
special care must be taken to handle the term $\gamma\mathbf{m}^{n-1}\times(\mathbf{m}^{n-1}\times\Delta\widetilde{\mathbf{m}}^n)$. In fact, this issue should be addressed by taking an inner product with the error evolutionary equation by \(\Delta \widetilde{\mathbf{e}}^n\), where $\widetilde{\mathbf{e}}^n = \mathbf{m}_e^n - \widetilde{\mathbf{m}}^n$, $\mathbf{e}^n = \mathbf{m}_e^n - \mathbf{m}^n$, $\mathbf{m}_e^n$, and $\mathbf{m}^n$, $\widetilde{\mathbf{m}}^n$ stand for the exact and numerical solutions, respectively. However, the best error estimate at the renormalization stage \cite{akrivis2021higher, gui2022convergence}, in the gradient norm, is stated as 
\begin{equation} 
  \setlength{\abovedisplayskip}{4pt}   
  \|\nabla \mathbf{e}^n\| \leq \|\nabla \widetilde{\mathbf{e}}^n\|  + C \|\widetilde{\mathbf{e}}^n\| .  
\setlength{\belowdisplayskip}{4pt}  
  \label{gradient error est-previous-0} 
\end{equation} 
However, an $H^1$ convergence analysis could hardly go through based on this renormalization estimate, which comes from the fact that a temporal differentiation has to be taken into consideration in the renormalization stage. In fact, the constant appearing in~\eqref{gradient error est-previous-0} has to be sharpened to $C \Delta t$ to make a closed $H^1$ error estimate theoretically available. This subtle issue indicates an essential difficulty in the convergence analysis and error estimate for the projection-style method. 

As an alternate approach to overcome this subtle difficulty, if we just simply take an inner product with \(\widetilde{\mathbf{e}}^n\), an error bound associated with the term $\gamma\mathbf{m}^{n-1}\times(\mathbf{m}^{n-1}\times\Delta\widetilde{\mathbf{m}}^n)$ becomes another challenging issue, due to the highly complicated nonlinear structure, and insufficient diffusion coefficient to bound the nonlinear errors. In this work, we first give the weak form for the original finite difference scheme, derive a relation between interpolation and discrete gradient operators, and use the discrete summation by parts to establish an equivalent reformulation to avoid the highly complicated term. In fact, the point-wise length preserving feature of the numerical solution plays an essential role in the derivation of the equivalent weak form. Based on such a reformulation, a nonlinear Laplacian estimate is avoided, and the rest nonlinear error bounds could be derived with the help of discrete Sobolev interpolation, as well as a Law-of-Cosine style estimate of the numerical errors at the renormalization stage. These estimates in the reformulated weak form in turn lead to a theoretical justification of an optimal convergence rate. 
Moreover, in comparison with the convergence analysis presented in \cite{an2021optimal, gui2022convergence}, a few more refined estimates in terms of the intermediate numerical error induced by the spherical projection is reported in our work, in both the discrete \(L^2\) and \(H^1\) norms, which comes from the fact the intermediate stage numerical solution has a point-wise length greater than 1. 

In our knowledge, this numerical method is the first linear numerical scheme to the LLG equation that preserves the following combined theoretical properties: (i) point-wise length preservation, (ii) unconditional original energy dissipation, (iii) a theoretical justification of convergence analysis and optimal rate error estimate.



The paper is organized as follows. In Section 2, we review some preliminary notations and estimates. The fully discrete finite difference scheme is presented, and an unconditional energy dissipation and stability is established in Section 3. In Section 4, we provide a rigorous error analysis for the proposed numerical scheme. In Section 5, some numerical experiments are carried out using the constructed scheme. Finally, some concluding remarks are made in Section 6. 
  \section{Some preliminaries}
In this section, we review some preliminary notations and estimates, which will be used frequently in the later sections. Throughout the paper, we use $C$, with or without subscript, to denote a positive constant, which could have different values at different appearances.

We use the standard notations $L^2(\Omega)$, $H^k(\Omega)$ and $W^{k,p}(\Omega)$ to denote the usual Sobolev spaces over $\Omega$. In particular, $\| \cdot \|$ and $(\cdot,\cdot)$ are used to denote the norm and the inner product in $L^2(\Omega)$, respectively. The vectors and vector spaces will be indicated by boldface type.

For simplicity, we only consider the unit cube domain $\Omega = (0,1)^3$. Let $x_i = ih_x,\,y_j = jh_z,\,z_k = kh_z$ for $0 \leq i \leq N_x,\,0\leq j \leq N_y,\,0\leq k\leq N_z$ define a mesh on $\Omega$ with the mesh sizes $h_x = 1/N_x,\,h_y=1/N_y,$ and $h_z = 1/N_z$. Set $\Delta t = T/N,\,t^n = n\Delta t$ and let $\mathbf{m}^n_{i,j,k}$ be the numerical approximation of $\mathbf{m}(x_i,y_j,z_k,t^n)$. In this paper, we primarily consider periodic boundary conditions, which are specified as follows (in the $z$ direction):
\begin{equation}\label{boundarycondition}
  \setlength{\abovedisplayskip}{4pt}  
\begin{array}{ll}
     \mathbf{m}_{-1,j,k} = \mathbf{m}_{N_x-1,j,k}, \, \, \, \mathbf{m}_{N_x,j,k} = \mathbf{m}_{0,j,k}, \, \, \, 
    \mathbf{m}_{i,-1,k} = \mathbf{m}_{i,N_y-1,k}, \, \, \, \mathbf{m}_{i,N_y,k} = \mathbf{m}_{i,0,k}, \\
    \mathbf{m}_{i,j,-1} = \mathbf{m}_{i,j,N_z-1}, \, \, \, \mathbf{m}_{i,j,N_z} = \mathbf{m}_{i,j,0}. 
\end{array}
  \setlength{\belowdisplayskip}{4pt}  
\end{equation}
An extension to the case of homogeneous Neumann boundary condition would be straightforward. Furthermore, the computational grid and the associated grid function space are defined as 
\begin{equation*}
  \setlength{\abovedisplayskip}{4pt}  
\aligned
    \mathcal{T}_h &= \{(i,j,k)|\,i=0,\cdots,N_x-1,\,j=0,\cdots,N_y-1,\,k=0,\cdots,N_z-1\},\\
    \mathcal{M}_h &= \{m_{i,j,k}|\,(i,j,k) \in \mathcal{T}_h\},~~~~ \mathcal{M}^3_h = \mathcal{M}_h\times \mathcal{M}_h\times \mathcal{M}_h.
\endaligned
  \setlength{\belowdisplayskip}{4pt}  
\end{equation*}
Based on the above notations, for $f \in \mathcal{M}_h$, we denote
\begin{equation*}
  \setlength{\abovedisplayskip}{4pt}  
\begin{aligned} 
  & 
    \nabla^x_h f_{i+1/2,j,k} = \frac{f_{i+1,j,k} - f_{i,j,k}}{h_x},\,\nabla^y_h f_{i,j+1/2,k} = \frac{f_{i,j+1,k} - f_{i,j,k}}{h_y},\,\nabla^z_h f_{i,j,k+1/2} = \frac{f_{i,j,k+1} - f_{i,j,k}}{h_z},
\\
  & \mbox{and} \quad 
    \nabla_hf = (\nabla_h^x f,\,\nabla^y_hf,\,\nabla_h^zf)^T. 
\end{aligned} 
  \setlength{\belowdisplayskip}{4pt}  
\end{equation*}
Then for $\boldsymbol{f} = (f^x,f^y,f^z) \in \mathcal{M}^3$, the discrete gradient operator becomes 
\begin{equation*}
\setlength{\abovedisplayskip}{4pt}  
    \nabla_h\boldsymbol{f} = 
    \left[ \begin{array}{ccc}
         \nabla_h^x f^x& \nabla_h^y f^x & \nabla_h^zf^x\\
         \nabla_h^x f^y& \nabla_h^y f^y & \nabla_h^zf^y\\
         \nabla_h^x f^z& \nabla_h^y f^z & \nabla_h^zf^z
    \end{array}\right], 
\setlength{\belowdisplayskip}{4pt}  
\end{equation*}
and $\nabla_h\boldsymbol{f}:\nabla_h \boldsymbol{g}$ denotes the operation of multiplying the elements at the corresponding positions of the matrices and then summing them up. In turn, the standard second-order central difference operator turns out to be 
\begin{equation*}
\setlength{\abovedisplayskip}{4pt}   
    \Delta_h \boldsymbol{f}_{i,j,k} = \frac{\boldsymbol{f}_{i+1,j,k} - 2\boldsymbol{f}_{i,j,k} + \boldsymbol{f}_{i-1,j,k}}{h_x} + \frac{\boldsymbol{f}_{i,j+1,k} - 2\boldsymbol{f}_{i,j,k} + \boldsymbol{f}_{i,j-1,k}}{h_y} + \frac{\boldsymbol{f}_{i,j,k+1} - 2\boldsymbol{f}_{i,j,k} + \boldsymbol{f}_{i,j,k-1}}{h_z}. 
 \setlength{\belowdisplayskip}{4pt}  
\end{equation*}
Next, we introduce a discrete $\ell^2$ inner product and several discrete norms. For grid functions $\boldsymbol{f},\,\boldsymbol{g} \in \mathcal{M}^3_h$, the discrete $\ell^2$ inner product is defined as 
\begin{equation*}
  \setlength{\abovedisplayskip}{4pt}   
\aligned
    \langle \boldsymbol{f},\boldsymbol{g} \rangle &= h^3\sum_{\mathbf{I} \in \mathcal{T}_h} \boldsymbol{f_I} \cdot\boldsymbol{g_I}, \ 
    \langle \nabla_h \boldsymbol{f},\nabla_h\boldsymbol{g} \rangle &= h^3\sum_{\mathbf{I} \in \mathcal{T}_h} \nabla_h \boldsymbol{f_I} : \nabla_h\boldsymbol{g_I},
\endaligned
  \setlength{\belowdisplayskip}{4pt}  
\end{equation*}
where a uniform mesh of $h=h_x=h_y=h_z$ is considered, the discrete $\ell^2$ norm is given by 
\begin{equation*}
   \setlength{\abovedisplayskip}{4pt}  
     \|\boldsymbol{f}\|^2_{2} = h^3\sum_{\mathbf{I} \in \mathcal{T}_h} |\boldsymbol{f_I}|^2,~~\|\nabla_h\boldsymbol{f}\|^2_{2} = h^3\sum_{\mathbf{I} \in \mathcal{T}_h} \nabla_h \boldsymbol{f_I} : \nabla_h\boldsymbol{f_I} . 
     \setlength{\belowdisplayskip}{4pt}  
\end{equation*}
Subsequently, the discrete $\ell^p$ norm ($1 \leq p <\infty$) becomes 
 \begin{equation*} 
   \setlength{\abovedisplayskip}{4pt}  
   \|\boldsymbol{f}\|^p_{p} =  h^3\sum_{\mathbf{I} \in \mathcal{T}_h} |\boldsymbol{f_I}|^p,~~\|\nabla_h\boldsymbol{f}\|^p_{p} = h^3\sum_{\mathbf{I} \in \mathcal{T}_h} |\nabla_h \boldsymbol{f_I} : \nabla_h\boldsymbol{f_I}|^{p/2}. 
   \setlength{\belowdisplayskip}{4pt}  
\end{equation*} 
Furthermore, the discrete $H^1_h$ norm of the grid function $\boldsymbol{f}$ is introduced as 
\begin{equation*}
   \setlength{\abovedisplayskip}{4pt}  
    \|\boldsymbol{f}\|^2_{H^1_h} = \|\boldsymbol{f}\|^2_{2} + \|\nabla_h\boldsymbol{f}\|^2_{2} . 
   \setlength{\belowdisplayskip}{4pt}  
\end{equation*}
Meanwhile, the discrete $\ell^\infty$ norm can be defined as $\|\boldsymbol{f}\|_{\infty} = \max\limits_{\mathbf{I} \in \mathcal{T}_h}|\boldsymbol{f}|$.

The following lemmas will be frequently used in the subsequent analysis. Their proofs and more detailed discussions could be found in \cite{bao2013optimal, samarskii1976difference, yu1995general}, etc. 
 \medskip
 \begin{lemma}(Summation by parts)
     For any grid functions $\boldsymbol{f},\,\boldsymbol{g} \in \mathcal{M}_h^3$ with the discrete boundary condition \eqref{boundarycondition}, the following identity holds:
     \begin{equation}
       \setlength{\abovedisplayskip}{4pt}  
         -\langle \Delta_h \boldsymbol{f},\boldsymbol{g} \rangle 
         = \langle \nabla_h\boldsymbol{f},\nabla_h\boldsymbol{g} \rangle .  \label{summation-1} 
      \setlength{\belowdisplayskip}{4pt}  
     \end{equation}
 \end{lemma}
 \begin{lemma}(Inverse and interpolation inequalities)\label{leminverse}
     For any grid function $\boldsymbol{f} \in \mathcal{M}_h^3$, we have 
     \begin{align}
       \setlength{\abovedisplayskip}{4pt}  
     \displaystyle
         \|\boldsymbol{f}\|_{q} &\leq \breve{C}_0 h^{-3 (\frac{1}{p} - \frac{1}{q})}\|\boldsymbol{f}\|_{p},~~1\leq p < q  \leq \infty,  \label{inverse-1} \\
         \displaystyle
         \|\boldsymbol{f}\|_{q} &\leq \breve{C}_0 \|\boldsymbol{f}\|_{2}^{\frac{6-q}{2q}} \cdot \|\boldsymbol{f}\|^{\frac{3q-6}{2q}}_{H^1_h},~~\forall\,2 \leq q \leq 6 ,  \label{interpolation-1} 
         \setlength{\belowdisplayskip}{4pt}  
     \end{align}
     in which $\breve{C}_0$ only depends on $\Omega$. 
 \end{lemma} 
 
 \begin{proof} 
 The proof of the inverse inequality~\eqref{inverse-1} is straightforward. The proof for the interpolation inequality~\eqref{interpolation-1} comes from a combination of the discrete H\"older interpolation inequality 
 \begin{equation} 
    \setlength{\abovedisplayskip}{4pt}  
   \|\boldsymbol{f}\|_{q} \leq \|\boldsymbol{f}\|_{2}^{\frac{6-q}{2q}} 
   \cdot \|\boldsymbol{f}\|^{\frac{3q-6}{2q}}_6 ,~~\forall\,2 \leq q \leq 6 ,  \label{interpolation-2} 
         \setlength{\belowdisplayskip}{4pt} 
\end{equation} 
and the discrete Sobolev embedding 
\begin{equation} 
  \setlength{\abovedisplayskip}{4pt}  
   \|\boldsymbol{f}\|_6 \le C \|\boldsymbol{f}\|_{H_h^1}  ,   \label{interpolation-3} 
         \setlength{\belowdisplayskip}{4pt} 
\end{equation}
see Lemma 6 in an existing work~\cite{chen20b}.
 \end{proof} 
 
 For any grid functions $f,\,g$ and vector grid functions $\boldsymbol{f},\,\boldsymbol{g}$, the definition of $\nabla_h$ implies the following equalities:
 \begin{equation} 
 \begin{aligned}
    \setlength{\abovedisplayskip}{4pt}  
     \nabla_h (fg) &= \nabla_hf\,\Pi_hg + \nabla_hg\,\Pi_hf,\ \ 
     \nabla_h(\boldsymbol{f} \cdot \boldsymbol{g}) = \nabla_h\boldsymbol{f} \cdot \Pi_h\boldsymbol{g} + \nabla_h \boldsymbol{g} \cdot \Pi_h\boldsymbol{f},\\
     \nabla_h(f\boldsymbol{g}) &= \nabla_hf\,\Pi_h\boldsymbol{g} + \nabla_h\boldsymbol{g}\Pi_hf,\ \ 
     \nabla_h(\boldsymbol{f}\times\boldsymbol{g}) = \nabla_h\boldsymbol{f}\times\Pi_h\boldsymbol{g} + \Pi_h\boldsymbol{f}\times \nabla_h\boldsymbol{g},
     \setlength{\belowdisplayskip}{4pt}  
 \end{aligned}
   \label{product gradient-1} 
\end{equation} 
where \(\Pi_h\) is the matched central interpolation operator and the operations follow the tensor form.
 
The following discrete version of the Gronwall lemma (see \cite{HeSu07,shen1990long}) will be frequently used. 
\medskip
\begin{lemma} \label{lem: gronwall2}
Let $a_k$, $b_k$, $c_k$, $d_k$, $\gamma_k$, $\Delta t_k$ be non negative real numbers such that
\begin{equation*}\label{e_Gronwall3}
  \setlength{\abovedisplayskip}{4pt}  
a_{k+1}-a_k+b_{k+1}\Delta t_{k+1}+c_{k+1}\Delta t_{k+1}-c_k\Delta t_k\leq a_kd_k\Delta t_k+\gamma_{k+1}\Delta t_{k+1} , \quad \forall \, \, 0\leq k\leq m . 
  \setlength{\belowdisplayskip}{4pt}  
\end{equation*}
Then we have 
 \begin{equation*}\label{e_Gronwall4}
  \setlength{\abovedisplayskip}{4pt}  
a_{m+1}+\sum_{k=0}^{m+1}b_k\Delta t_k \leq \exp \Big(\sum_{k=0}^md_k\Delta t_k \Big) \{a_0+(b_0+c_0)\Delta t_0+\sum_{k=1}^{m+1}\gamma_k\Delta t_k \}.
  \setlength{\belowdisplayskip}{4pt}  
\end{equation*}
\end{lemma}
\section{Semi-implicit projection scheme}
In this section, we present the fully discrete finite difference scheme based on the semi-implicit projection method, and derive an unconditional and original energy dissipation. To facilitate the theoretical analysis, before presenting the fully discrete version, we first derive an equivalent form of the semi-discrete scheme \eqref{semidiscrete}.

\textbf{Scheme \uppercase\expandafter{\romannumeral1} (Semi-discrete scheme in time)}. The semi-implicit projection  (SIP) scheme \eqref{semidiscrete} could be rewritten as follows.  
\begin{equation}\label{semi-discretescheme}
  \setlength{\abovedisplayskip}{4pt} 
    \displaystyle
    \left\{
    \begin{array}{l}
    \displaystyle
         \frac{\widetilde{\mathbf{m}}^n - \mathbf{m}^{n-1}}{\Delta t} = -\beta \mathbf{m}^{n-1}\times\Delta\widetilde{\mathbf{m}}^{n}-\gamma\mathbf{m}^{n-1}\times(\mathbf{m}^{n-1}\times\Delta\widetilde{\mathbf{m}}^n),  \\
         \displaystyle
          \mathbf{m}^n = \frac{\widetilde{\mathbf{m}}^{n}}{|\widetilde{\mathbf{m}}^{n}|}.
    \end{array} \right. 
    \setlength{\belowdisplayskip}{4pt} 
\end{equation}
An unconditional energy dissipation estimate is stated below; the proof could be found in  \cite{du2025semi} . 
\medskip
\begin{theorem}  \cite{du2025semi}  \label{thm1}
The semi-discrete scheme \eqref{semi-discretescheme} is unconditionally dissipative with respect to the original energy:
    \begin{equation} \label{eq_final_Stability}
     \setlength{\abovedisplayskip}{4pt} 
        \frac{1}{2}\|\nabla\mathbf{m}^n\|^2 - \frac{1}{2}\|\nabla\mathbf{m}^{n-1}\|^2 \leq -\Delta t\gamma\|\mathbf{m}^{n-1} \times \Delta\widetilde{\mathbf{m}}^n\|^2. 
       \setlength{\belowdisplayskip}{4pt} 
    \end{equation}
\end{theorem}

Now, we present the finite difference fully discrete scheme that preserves both a point-wise length, $|\textbf{m}^n| = 1$, and the original energy dissipation.

\textbf{Scheme \uppercase\expandafter{\romannumeral2} (Fully discrete scheme)}. The fully discrete finite difference scheme, based on the semi-discrete approach \eqref{semi-discretescheme}, could be written as follows:
\begin{equation}\label{discretescheme}
    \setlength{\abovedisplayskip}{4pt} 
    \left\{
    \begin{array}{l}
    \displaystyle
         \frac{\widetilde{\mathbf{m}}^n - \mathbf{m}^{n-1}}{\Delta t} = -\beta \mathbf{m}^{n-1}\times\Delta_h\widetilde{\mathbf{m}}^{n}-\gamma\mathbf{m}^{n-1}\times(\mathbf{m}^{n-1}\times\Delta_h\widetilde{\mathbf{m}}^n),  \\
         \displaystyle
          \mathbf{m}^n = \frac{\widetilde{\mathbf{m}}^{n}}{|\widetilde{\mathbf{m}}^{n}|}.
    \end{array} \right. 
    \setlength{\belowdisplayskip}{4pt} 
\end{equation}

The unique solvability of this linear numerical scheme is stated in the following theorem. 
\medskip
\begin{theorem} \label{thm: solvability} 
Given $\mathbf{m}^{n-1}$, there is a unique solution $(\widetilde{\mathbf{m}}^n , \mathbf{m}^n)$ to the numerical system~\eqref{discretescheme}. 
\end{theorem}

\begin{proof} 
Given $\mathbf{m}^{n-1}$, the intermediate numerical stage in \eqref{discretescheme} corresponds to a non-constant coefficient (time dependent) linear system for $\widetilde{\mathbf{m}}^{n}$. In turn, the unique solvability is equivalent to the verification that the homogeneous numerical system (by throwing all the terms independent of the unknown numerical variable) would only have a trivially zero solution. In more details, the homogeneous version of the intermediate numerical stage turns out to be 
\begin{equation} 
   \setlength{\abovedisplayskip}{4pt} 
   (\Delta t)^{-1} \widetilde{\mathbf{m}}^n  = -\beta \mathbf{m}^{n-1}\times\Delta_h 
   \widetilde{\mathbf{m}}^{n}-\gamma\mathbf{m}^{n-1}\times(\mathbf{m}^{n-1}\times\Delta_h\widetilde{\mathbf{m}}^n) .  \label{solvability-1} 
    \setlength{\belowdisplayskip}{4pt} 
\end{equation}
Subsequently, a discrete inner product with $-\Delta_h \widetilde{\mathbf{m}}^n$ yields 
\begin{equation} 
   \setlength{\abovedisplayskip}{4pt} 
   (\Delta t)^{-1} \| \nabla_h \widetilde{\mathbf{m}}^n \|_2^2 = 
   \gamma \langle \mathbf{m}^{n-1} \times(\mathbf{m}^{n-1} \times \Delta_h\widetilde{\mathbf{m}}^n) , 
    \Delta_h \widetilde{\mathbf{m}}^n \rangle \le 0 , \label{solvability-2} 
    \setlength{\belowdisplayskip}{4pt} 
\end{equation}
in which the summation by parts formula~\eqref{summation-1} and the $\ell^2$ orthogonality between $\Delta_h \widetilde{\mathbf{m}}^n$ and the first term on the right hand side of~\eqref{solvability-1} have been applied. As a result, the only possible solution to~\eqref{solvability-2} would be $\widetilde{\mathbf{m}}^n \equiv K_0$, a constant. Furthermore, a substitution of this constant value into~\eqref{solvability-1} implies that $\widetilde{\mathbf{m}}^n \equiv 0$. Therefore, the unique solvability for the numerical solution $\widetilde{\mathbf{m}}^n$ has been proved. As a consequence, given $\widetilde{\mathbf{m}}^n$, the unique solvability of the projection solution $\mathbf{m}^n$ is more straightforward. This finishes the proof of Theorem~\ref{thm: solvability}. 
\end{proof}

The following lemma is needed to establish an energy dissipation of the numerical scheme \eqref{discretescheme}. 
\medskip
\begin{lemma}\label{lem2}
    Let $\{\mathbf{m}^n,\widetilde{\mathbf{m}}^{n}\}$ be the solution of \eqref{discretescheme} with $ |\widetilde{\mathbf{m}}^{n}|\geq1 $ and $ | \mathbf{m}^{n}|= 1 $. For any $1\leq n \leq N$, there holds
    \begin{equation}\label{lem2eq1}
    \setlength{\abovedisplayskip}{4pt} 
        | \nabla_h\mathbf{m}^n | \leq | \nabla_h\widetilde{\mathbf{m}}^n | ,  \quad 
        \mbox{at a point-wise level}\textcolor{black}{,}
    \setlength{\belowdisplayskip}{4pt} 
    \end{equation}
    which in turn implies that $\|\nabla_h\mathbf{m}^n\|^2_2 \leq \|\nabla_h\widetilde{\mathbf{m}}^n\|^2_{2}$.
\end{lemma}
\begin{proof}
    The fully discrete scheme defined by \eqref{discretescheme} is equivalent to
    \begin{equation*} 
    \setlength{\abovedisplayskip}{4pt} 
         \widetilde{\mathbf{m}}^n = \mathbf{m}^{n-1}-\Delta t \mathbf{m}^{n-1}\times(\beta\Delta_h\widetilde{\mathbf{m}}^{n}+\gamma(\mathbf{m}^{n-1}\times\Delta_h\widetilde{\mathbf{m}}^n)) ,  
         \quad \mbox{with} \, \, \, \mathbf{m}^n = \widetilde{\mathbf{m}}^n/|\widetilde{\mathbf{m}}^n| . 
    \setlength{\belowdisplayskip}{4pt} 
    \end{equation*}
Since the two vectors on the right hand are orthogonal, it is clear that 
    \begin{equation*} 
    \setlength{\abovedisplayskip}{4pt} 
        |\widetilde{\mathbf{m}}^n|^2 = {\color{black}|\mathbf{m}^{n-1}|^2} + \Delta t^2|\mathbf{m}^{n-1}\times(\beta\Delta_h\widetilde{\mathbf{m}}^{n}+\gamma(\mathbf{m}^{n-1}\times\Delta_h\widetilde{\mathbf{m}}^n))|^2 \geq 1.  
    \setlength{\belowdisplayskip}{4pt} 
    \end{equation*}
Next, we take \(|\nabla_h^x \mathbf{m}^n|^2\) as an example, and the same argument could be applied to the gradient in the other directions. The finite difference in the $x$ direction takes a form of 
    \begin{equation*} 
    \setlength{\abovedisplayskip}{4pt} 
        |\nabla^x_h\mathbf{m}^n_{i+1/2,j,k}|^2 = |h^{-1}(\mathbf{m}^n_{i+1,j,k} - \mathbf{m}^n_{i,j,k})|^2. 
    \setlength{\belowdisplayskip}{4pt} 
    \end{equation*}
    Set $\theta = \cos^{-1} \langle\mathbf{m}^n_{i+1,j,k},~\mathbf{m}^n_{i,j,k}\rangle,\,0 \leq \theta \leq \pi,~ \alpha_1 = |\widetilde{\mathbf{m}}^{n}_{i+1,j,k}|\geq1,\,\alpha_2 = |\widetilde{\mathbf{m}}^{n}_{i,j,k}| \geq 1$. An application of the law of cosines (at the point-wise level) reveals that 
    \begin{equation}\label{lem2eq2} 
    \setlength{\abovedisplayskip}{4pt} 
        h^2|\nabla^x_h\widetilde{\mathbf{m}}_{i+1/2}^n|^2  = \alpha_1^2 + \alpha_2^2 - 2\alpha_1\alpha_2\cos\theta, \quad 
        h^2|\nabla_h^x\mathbf{m}_{i+1/2}^n|^2 = 2 - 2\cos\theta .
    \setlength{\belowdisplayskip}{4pt} 
    \end{equation}
    Taking the difference, and using the fact that $\alpha_i \geq 1$ ($i=1,2$), we arrive at 
    \begin{equation*}
    \setlength{\abovedisplayskip}{4pt} 
    \aligned
        &
         h^2(|\nabla_h^x\widetilde{\mathbf{m}}^n_{i+1/2}|^2 - |\nabla_h^x\mathbf{m}^n_{i+1/2}|^2) \\
      =& (\alpha_1^2 - 1) + (\alpha^2_2 - 1)-2\cos\theta(\alpha_1\alpha_2-1) 
       \geq (\alpha_1 - \alpha_2)^2 \geq 0.
    \endaligned
    \setlength{\belowdisplayskip}{4pt} 
    \end{equation*}
    Therefore, a summation over all the directions yields the point-wise result \eqref{lem2eq1}.
\end{proof}

  Now we derive the original energy dissipation law for the fully discrete scheme \eqref{discretescheme}.
\medskip
\begin{theorem}\label{discretethm1}
    Let $\{\mathbf{m}^n,\widetilde{\mathbf{m}}^{n}\}$ be the solution of the fully discrete finite difference scheme \eqref{discretescheme}. The following original energy dissipation law is preserved, for any $1\leq n \leq N$
     \begin{equation} \label{thm3_original energy result} 
     \setlength{\abovedisplayskip}{4pt} 
     \frac{1}{2}\|\nabla_h \mathbf{m}^{n}\|_{2}^2 - \frac{1}{2}\|\nabla_h \mathbf{m}^{n-1}\|^2_{2}\leq  -\gamma\Delta t\|\mathbf{m}^{n-1} \times \Delta_h\widetilde{\mathbf{m}}^n\|^2_{2} \leq 0. 
     \setlength{\belowdisplayskip}{4pt} 
 \end{equation}
    In turn, the following unconditional $H_h^1$ bound becomes available: 
    \begin{equation}\label{thm3.2eq1} 
      \setlength{\abovedisplayskip}{4pt} 
         \frac{1}{2}\|\nabla_h\mathbf{m}^n\|^2_{2}+ \gamma\Delta t\sum^{n}_{k=1}\|\mathbf{m}^{k-1} \times\Delta_h\widetilde{\mathbf{m}}^k\|_{2}^2 \leq \frac{1}{2}\|\nabla_h\mathbf{m}^0\|^2_{2}. 
         \setlength{\belowdisplayskip}{4pt} 
    \end{equation}    
\end{theorem}
\begin{proof}
    Taking the inner product with \eqref{discretescheme} by $-\Delta \widetilde{\mathbf{m}}^n$, and making use of an identity
    \begin{equation*}
	\aligned
	(\textbf{a} \times \textbf{b})\cdot\textbf{c} = -(\textbf{a} \times \textbf{c})\cdot \textbf{b}, \ \ \textbf{a},\textbf{b}, \textbf{c} \in \mathbb{R}^3,
	\endaligned
    \end{equation*}
we obtain
 \begin{equation*}
     \setlength{\abovedisplayskip}{4pt} 
     \frac{1}{2}\|\nabla_h \mathbf{m}^{n}\|_{2}^2 \leq \frac{1}{2}\|\nabla_h \widetilde{\mathbf{m}}^{n}\|_{2}^2  \leq \frac{1}{2}\|\nabla \mathbf{m}^{n-1}\|^2_{2} -\gamma\Delta t\|\mathbf{m}^{n-1} \times \Delta_h\widetilde{\mathbf{m}}^n\|^2_{2}, 
     \setlength{\belowdisplayskip}{4pt} 
 \end{equation*}
 in which Lemma \ref{lem2} has been applied in the first step. A summation from $k=1$ to $k=n$) of the above inequality leads to \eqref{thm3_original energy result}, and the discrete $H^1$ bound~\eqref{thm3.2eq1} becomes a direct consequence.
\end{proof}

\section{Error estimate}
In this section, we present a rigorous error analysis of the fully discrete scheme \eqref{discretescheme}. This analysis is nontrivial due to the highly complicated nonlinear structure inherent in the constructed scheme. In the theoretical justification of the error estimate, special care must be taken to handle the error term associated with $\gamma\mathbf{m}^{n-1}\times(\mathbf{m}^{n-1}\times\Delta\widetilde{\mathbf{m}}^n)$. To accomplish this goal, we have to derive an equivalent weak form of the original numerical scheme, which comes from the point-wise length preservation, and the optimal error estimate could be established based on this reformulation. After presenting the complete theoretical proof, we will elaborate on the necessity of employing the equivalent weak form to enhance the understanding of this method.
\subsection{Proof process}
With the help of an identity
   \begin{equation}\label{e_curl operator}
   \setlength{\abovedisplayskip}{4pt} 
	\textbf{a} \times ( \textbf{b} \times \textbf{c} ) = (\textbf{a} \cdot \textbf{c}) \textbf{b} - (\textbf{a} \cdot \textbf{b}) \textbf{c}, \ \ \textbf{a},\textbf{b}, \textbf{c} \in \mathbb{R}^3,
   \setlength{\belowdisplayskip}{4pt} 
    \end{equation}
it is clear that the original numerical scheme \eqref{discretescheme} could be equivalently rewritten as 
\begin{equation}\label{discrete}
    \setlength{\abovedisplayskip}{4pt} 
    \displaystyle
         \frac{\widetilde{\mathbf{m}}^n - \mathbf{m}^{n-1}}{\Delta t} -\gamma \Delta_h\widetilde{\mathbf{m}}^n= -\beta \mathbf{m}^{n-1}\times\Delta_h\widetilde{\mathbf{m}}^{n}-\gamma(\mathbf{m}^{n-1}\cdot\Delta_h\widetilde{\mathbf{m}}^{n})\mathbf{m}^{n-1},  \quad 
         \mbox{with} \, \, \,  \mathbf{m}^n = \frac{\widetilde{\mathbf{m}}^{n}}{|\widetilde{\mathbf{m}}^{n}|} . 
         \setlength{\belowdisplayskip}{4pt} 
\end{equation}

For any vector function $\mathbf{v} \in \mathcal{M}_h^3$, a discrete inner product with \eqref{discrete} by $\mathbf{v}$ leads to
\begin{equation}\label{discreteip}
     \setlength{\abovedisplayskip}{4pt} 
    \aligned
    &\Big\langle \frac{\widetilde{\mathbf{m}}^{n} - \mathbf{m}^{n-1}}{\Delta t},\mathbf{v} \Big\rangle + \gamma \langle \nabla_h\widetilde{\mathbf{m}}^n,\nabla_h\mathbf{v} \rangle \\
    = &\beta \langle \nabla_h\widetilde{\mathbf{m}},\nabla_h(\mathbf{v} \times \mathbf{m}^{n-1}) \rangle 
    -\gamma \langle \mathbf{m}^{n-1}\cdot\Delta_h\mathbf{m}^{n-1},\mathbf{m}^{n-1}\cdot\mathbf{v} \rangle \\
    & -\gamma \langle \mathbf{m}^{n-1}\cdot\Delta_h(\widetilde{\mathbf{m}}^n-\mathbf{m}^{n-1}),\mathbf{m}^{n-1}\cdot\mathbf{v} \rangle \\
    = & \beta \langle \nabla_h\widetilde{\mathbf{m}}^n,\nabla_h\mathbf{v} \times \Pi_h\mathbf{m}^{n-1} \rangle 
    + \beta \langle \nabla_h\widetilde{\mathbf{m}}^n,\Pi_h\mathbf{v} \times \nabla_h\mathbf{m}^{n-1} \rangle \\
    & +\gamma \langle \nabla_h\mathbf{m}^{n-1},\nabla_h(\mathbf{m}^{n-1}\cdot\mathbf{v})\Pi_h\mathbf{m}^{n-1} + \Pi_h(\mathbf{m}^{n-1}\cdot\mathbf{v})\nabla_h\mathbf{m}^{n-1} \rangle  \\
    & + \gamma \langle \nabla_h(\widetilde{\mathbf{m}}^n-\mathbf{m}^{n-1}),\nabla_h(\mathbf{m}^{n-1}\cdot\mathbf{v})\Pi_h\mathbf{m}^{n-1} 
    + \Pi_h(\mathbf{m}^{n-1}\cdot\mathbf{v})\nabla_h\mathbf{m}^{n-1} \rangle , 
    \endaligned
    \setlength{\belowdisplayskip}{4pt} 
\end{equation}
in which the summation by parts formula~\eqref{summation-1} and the discrete gradient expansion formula~\eqref{product gradient-1} have been repeated applied. Moreover, a few interpolation and discrete gradient equalities are needed. 
\medskip
\begin{lemma}\label{lem4}
    Suppose $\{\mathbf{m}^n,\widetilde{\mathbf{m}}^{n}\}$ be the solution of \eqref{discretescheme} with \(|\mathbf{m}^{n-1}|^2=1\). Then we have
    \begin{align}
      \setlength{\abovedisplayskip}{4pt} 
        \Pi_h\mathbf{m}^{n-1} \cdot(\nabla_h \mathbf{m}^{n-1} ) &= 0,\label{lem4eq1}\\
        \Pi_h\mathbf{m}^{n-1} \cdot \nabla_h(\widetilde{\mathbf{m}}^n -\mathbf{m}^{n-1}) &= - \Pi_h(\widetilde{\mathbf{m}}^n -\mathbf{m}^{n-1})\cdot\nabla_h\mathbf{m}^{n-1} . \label{lem4eq2} 
      \setlength{\belowdisplayskip}{4pt} 
    \end{align}
\end{lemma}
\begin{proof}
    Based on the point-wise length preservation $|\mathbf{m}^{n-1} |=1$, we see that 
    \begin{equation*} 
      \setlength{\abovedisplayskip}{4pt} 
        \Pi_h\mathbf{m}^{n-1}_{i+1/2, j,k} \cdot \nabla_h^x\mathbf{m}^{n-1}_{i+1/2, j,k} = \frac{\mathbf{m}^{\color{black}n-1} _{i+1,j,k}+\mathbf{m}^{\color{black}n-1}_{i,j,k}}{2} \cdot\frac{\mathbf{m}^{\color{black}n-1}_{i+1,j,k}-\mathbf{m}^{\color{black}n-1}_{i,j,k}}{h} 
        =0.
     \setlength{\belowdisplayskip}{4pt} 
    \end{equation*}
The same argument could be applied in the $y$ and $z$ directions. As a result, the desired result \eqref{lem4eq1} becomes valid. Furthermore, taking a point-wise dot product with \eqref{discretescheme} by \(\mathbf{m}^{n-1}\) gives 
    \begin{equation*} 
       \setlength{\abovedisplayskip}{4pt} 
        \widetilde{\mathbf{m}}^n\cdot\mathbf{m}^{n-1} = |\mathbf{m}^{n-1}|^2 = 1. 
        \setlength{\belowdisplayskip}{4pt} 
    \end{equation*}
With the help of the discrete gradient expansion formula~\eqref{product gradient-1}, the following identity is clear 
    \begin{equation*} 
    \setlength{\abovedisplayskip}{4pt} 
        \nabla_h(\mathbf{\widetilde{\mathbf{m}}}^n\cdot\mathbf{m}^{n-1} - \mathbf{m}^{n-1}\cdot\mathbf{m}^{n-1}) = \Pi_h\mathbf{m}^{n-1} \cdot \nabla_h(\widetilde{\mathbf{m}}^n -\mathbf{m}^{n-1}) + \Pi_h(\widetilde{\mathbf{m}}^n -\mathbf{m}^{n-1})\cdot\nabla_h\mathbf{m}^{n-1} = 0 ,   
    \setlength{\belowdisplayskip}{4pt} 
    \end{equation*}
which in turn yields the desired result \eqref{lem4eq2}. This completes the proof of Lemma~\ref{lem4}. 
\end{proof}

Based on the identities \eqref{lem4eq1} and \eqref{lem4eq2}, we see that the weak form \eqref{discreteip} could be simplified as 
\begin{equation}\label{m_hconsistency}
    \setlength{\abovedisplayskip}{4pt} 
    \aligned
    & \Big\langle \frac{\widetilde{\mathbf{m}}^{n} - \mathbf{m}^{n-1}}{\Delta t},\mathbf{v} \Big\rangle 
     + \gamma \langle \nabla_h\widetilde{\mathbf{m}}^n,\nabla_h\mathbf{v} \rangle  
    = \beta \langle \nabla_h\widetilde{\mathbf{m}}^n,\nabla_h\mathbf{v} \times \Pi_h\mathbf{m}^{n-1} \rangle + \beta \langle \nabla_h\widetilde{\mathbf{m}}^n,\Pi_h\mathbf{v} \times \nabla_h\mathbf{m}^{n-1} \rangle \\
    & \quad +\gamma \langle \nabla_h\widetilde{\mathbf{m}}^{n}:\nabla_h\mathbf{m}^{n-1},\Pi_h(\mathbf{m}^{n-1}\cdot\mathbf{v}) \rangle -\gamma \langle \Pi_h(\widetilde{\mathbf{m}}^n-\mathbf{m}^{n-1})\cdot\nabla_h\mathbf{m}^{n-1},\nabla_h(\mathbf{m}^{n-1}\cdot\mathbf{v}) \rangle .
    \endaligned
     \setlength{\belowdisplayskip}{4pt} 
\end{equation}
Meanwhile, we denote $\mathbf{m}_e$ as the exact solution for the PDE system \eqref{e_original model}. A careful Taylor expansion gives the truncation error and consistency estimate:
\textcolor{black}{
\begin{align} 
  \setlength{\abovedisplayskip}{4pt} 
  & 
    \frac{\mathbf{m}^n_e - \mathbf{m}^{n-1}_e}{\Delta t} -\gamma \Delta_h\mathbf{m}^n_e = -\beta \mathbf{m}^{n-1}_e\times\Delta_h\mathbf{m}^n_e-\gamma(\mathbf{m}^{n-1}_e\cdot\Delta_h\mathbf{m}^{n}_e)\mathbf{m}^{n-1}_e - \mathbf{R}^n , \label{consistency} 
\\
  & \mbox{with} \quad 
    \mathbf{R}^n = \frac{\partial \mathbf{m}^{n}_e}{\partial t} - \frac{\mathbf{m}^{n}_e - \mathbf{m}^{n-1}_e}{\Delta t} + \beta ( \mathbf{m}^{n}_e\times\Delta\mathbf{m}^{n}_e - \mathbf{m}^{n-1}_e \times \Delta_h\mathbf{m}^n_e ) \nonumber \\
    & \qquad \qquad \qquad 
    + \gamma ( \mathbf{m}^{n}_e\times(\mathbf{m}^{n}_e\times\Delta\mathbf{m}^{n}_e) 
    - \mathbf{m}^{n-1}_e\times(\mathbf{m}^{n-1}_e \times \Delta_h\mathbf{m}^{n}_e) ) \nonumber \\
    &\qquad \qquad \qquad 
    +\gamma \Delta_h(\mathbf{m}^{n-1}_e\cdot(\mathbf{m}^n_e - \mathbf{m}^{n-1}_e))\mathbf{m}^{n-1}_e\nonumber,
\\
  & 
   \Delta t \sum^{n}_{k=1}\|\mathbf{R}^k\|^2_{2} \leq C(\|\mathbf{m}_e\|_{L^{\infty}(0,T;H^{4}(\Omega))},\|\partial_t\mathbf{m}_e\|_{L^\infty(0,T;H^2(\Omega))},\|\partial_{tt}\mathbf{m}_e\|_{L^2(0,T;L^2(\Omega))})(h^4 + \Delta t^2 ).  \nonumber 
   \setlength{\belowdisplayskip}{4pt} 
\end{align} 
}
Similarly to \eqref{m_hconsistency}, for any vector function $\mathbf{v} \in \mathcal{M}_h^3$, taking a discrete inner product with \eqref{consistency} by a test function $\mathbf{v}$ leads to
\begin{equation}\label{mconsistency}
    \setlength{\abovedisplayskip}{4pt} 
    \aligned
    & \Big\langle \frac{\mathbf{m}_e^{n} - \mathbf{m}_e^{n-1}}{\Delta t},\mathbf{v} \Big\rangle 
     + \gamma \langle \nabla_h \mathbf{m}_e^n, \nabla_h\mathbf{v} \rangle  
    = \beta \langle \nabla_h \mathbf{m}_e^n, \nabla_h \mathbf{v} \times \Pi_h\mathbf{m}_e^{n-1} \rangle 
    + \beta \langle \nabla_h \mathbf{m}_e^n,\Pi_h \mathbf{v} \times \nabla_h \mathbf{m}_e^{n-1} \rangle \\
    &+\gamma \langle \nabla_h \mathbf{m}_e^{n}: \nabla_h \mathbf{m}_e^{n-1},\Pi_h(\mathbf{m}_e^{n-1} \cdot \mathbf{v}) \rangle -\gamma \langle \Pi_h( \mathbf{m}_e^n - \mathbf{m}_e^{n-1}) 
    \cdot \nabla_h \mathbf{m}_e^{n-1}, \nabla_h(\mathbf{m}_e^{n-1} \cdot \mathbf{v}) \rangle  \textcolor{black}{-\langle\mathbf{R}^n,\mathbf{v}\rangle} .
    \endaligned
     \setlength{\belowdisplayskip}{4pt} 
\end{equation}
Subsequently, the error functions are defined at a point-wise level:  
\begin{equation}  
   \setlength{\abovedisplayskip}{4pt} 
         \widetilde{\mathbf{e}}^k = \mathbf{m}^k_e - \widetilde{\mathbf{m}}^k , \, \, \,  
         \mathbf{e}^k = \mathbf{m}^k_e - \mathbf{m}^k ,  \quad \forall \, \, 0 \le k \le N . 
    \setlength{\belowdisplayskip}{4pt} 
\end{equation}
In turn, subtracting \eqref{m_hconsistency} from \eqref{mconsistency} gives an error evolutionary equation, in a weaker form: 
\begin{equation}\label{merror} 
    \setlength{\abovedisplayskip}{4pt} 
    \aligned
    & \langle \frac{\widetilde{\mathbf{e}}^{n} - \mathbf{e}^{n-1}}{\Delta t},\mathbf{v} \rangle 
    + \gamma \langle \nabla_h\widetilde{\mathbf{e}}^n,\nabla_h\mathbf{v} \rangle  
    =  \beta  \langle \nabla_h \widetilde{\mathbf{e}}^n ,\nabla_h\mathbf{v} \times \Pi_h\mathbf{m}^{n-1} \rangle 
    + \beta  \langle \nabla_h \widetilde{\mathbf{e}}^n ,\Pi_h\mathbf{v} \times \nabla_h\mathbf{m}^{n-1} \rangle \\
     &+ \beta \langle \nabla_h \mathbf{m}_e^n, \nabla_h\mathbf{v} \times \Pi_h\mathbf{e}^{n-1} \rangle 
     + \beta \langle \nabla_h \mathbf{m}_e^n, \Pi_h\mathbf{v} \times \nabla_h \mathbf{e}^{n-1} \rangle  \\
    &+\gamma  \langle \nabla_h \mathbf{m}_e^n : \nabla_h \mathbf{m}_e^{n-1} , 
    \Pi_h(\mathbf{e}^{n-1} \cdot\mathbf{v} ) \rangle 
    + \gamma \langle \nabla_h\widetilde{\mathbf{e}}^{n}: \nabla_h\mathbf{m}^{n-1} 
    + \nabla_h \mathbf{m}_e^{n}: \nabla_h\mathbf{e}^{n-1} , \Pi_h(\mathbf{m}^{n-1}\cdot\mathbf{v} ) \rangle  \\
    & - \gamma \langle ( \widetilde{\mathbf{e}}^n - \mathbf{e}^{n-1} ) \cdot \nabla_h \mathbf{m}^{n-1} , \nabla_h ( \mathbf{m}^{n-1} \cdot \mathbf{v} ) \rangle \textcolor{black}{- \gamma  \langle ( \mathbf{m}_e^n - \mathbf{m}_e^{n-1} ) \cdot \nabla_h \mathbf{e}^{n-1} , \nabla_h ( \mathbf{m}_{e}^{n-1} \cdot \mathbf{v} ) \rangle} \\
    & \textcolor{black}{- \gamma \langle (\mathbf{m}^{n}_e - \mathbf{m}^{n-1}_e)\cdot\nabla_h\mathbf{m}^{n-1},\nabla_h(\mathbf{e}^{n-1}\cdot\mathbf{v}) \rangle} + \langle \mathbf{R}^n, \mathbf{v} \rangle .
    \endaligned
    \setlength{\belowdisplayskip}{4pt} 
\end{equation}

    Two index values $p_0$ and $q_0$ are introduced to facilitate the nonlinear error estimates: 
\begin{equation} 
\setlength{\abovedisplayskip}{3pt}
  3 < p_0 = \frac{3}{1 - \frac14 \epsilon_0} , \, \, \,  q_0 = \frac{6}{1 + \frac12 \epsilon_0} < 6 , 
  \quad \mbox{so that}  \, \, \, \frac{1}{p_0} + \frac{1}{q_0} = \frac12 . 
\setlength{\belowdisplayskip}{3pt}
  \label{p-q-1} 
\end{equation} 
It is clear that $p_0$ is slightly greater than 3, and $q_0$ is slightly less than 6.

    Before deriving the error estimate, we first establish a more precise relationship between \(\widetilde{\mathbf{e}}^n\) and \(\mathbf{e}^n\). 
    \medskip
    \begin{lemma}\label{lem6}
        Suppose that $|\mathbf{m}^n| = 1$ and $|\widetilde{\mathbf{m}}^n| \geq 1,\,\,\forall\,\,1 \leq n \leq N$. The following inequality is valid:
        \begin{equation}\label{e_l^2}
        \setlength{\abovedisplayskip}{4pt}
            (i)\, | \mathbf{e}^n |^2 + | \widetilde{\mathbf{e}}^n - \mathbf{e}^n |^2 
            \leq  | \widetilde{\mathbf{e}}^n |^2 \le 2 ( | \mathbf{e}^n |^2 
            + | \widetilde{\mathbf{e}}^n - \mathbf{e}^n |^2 ) , \quad 
            \mbox{at a point-wise level} , 
        \setlength{\belowdisplayskip}{4pt}
        \end{equation}
        which implies that $\|\mathbf{e}^n\|_2^2 + \| \widetilde{\mathbf{e}}^n - \mathbf{e}^n \|_2^2 \leq \|\widetilde{\mathbf{e}}^n\|^2_{2} \le 2 ( \|\mathbf{e}^n\|_2^2 + \| \widetilde{\mathbf{e}}^n - \mathbf{e}^n \|_2^2)$. Furthermore, under the condition \(\|\widetilde{\mathbf{e}}^n\|_{\infty} \leq 1\), we have
        \begin{equation}\label{e_lh1}
          \setlength{\abovedisplayskip}{4pt}
            (ii)\,\|\nabla_h \mathbf{e}^n\|_{2} \leq M_0 (\|\widetilde{\mathbf{e}}^n\|_{2} + \|\nabla_h \widetilde{\mathbf{e}}^n\|_{2}), 
        \setlength{\belowdisplayskip}{4pt}
        \end{equation}
        where 
       \(M_0\) is a constant that only depends on $\Omega$ and the exact solution.
    \end{lemma}
    \begin{proof}
        Set $\theta = \cos^{-1} \langle\mathbf{m}^n,~\mathbf{m}_e^n\rangle,\,0 \leq \theta \leq \pi$, $ \alpha = |\widetilde{\mathbf{m}}^{n}| \geq 1$. A careful application of law of cosines implies that 
    \begin{equation}\label{lem6eq1}
     \setlength{\abovedisplayskip}{4pt}
        |\widetilde{\mathbf{e}}^n|^2 = \alpha^2 + 1 - 2\alpha\cos\theta, \quad 
        |\mathbf{e}^n|^2 = 2 - 2\cos\theta, \quad \alpha\geq 1.
    \setlength{\belowdisplayskip}{4pt}
    \end{equation}
    Taking the difference of \eqref{lem6eq1}, we obtain the left inequality in \eqref{e_l^2}: 
    \begin{equation*}
    \setlength{\abovedisplayskip}{4pt}
        (|\widetilde{\mathbf{e}}^n|^2 - |\mathbf{e}^n|^2)
         = \alpha^2-2(\alpha - 1)\cos \theta - 1\geq (\alpha - 1)^2 
         = |\widetilde{\mathbf{e}} - \mathbf{e}^n|^2 \geq 0 . 
    \setlength{\belowdisplayskip}{4pt}
    \end{equation*}
The right part of \eqref{e_l^2} comes from a direct application of Cauchy inequality. In terms of the gradient estimate, we take node $(i+1/2,j,k)$ as an example, and the same logic is applicable to the other directions.
    \begin{equation}\label{all}
    \setlength{\abovedisplayskip}{4pt}
        \aligned
        &\Big| \nabla_h \Big(\frac{\mathbf{m}^n_e}{|\mathbf{m}^n_e|} -  \frac{\widetilde{\mathbf{m}}^n}{|\widetilde{\mathbf{m}}^n|} \Big) \Big| = \Big| \nabla_h \Big( \mathbf{m}^n_e \Big( \frac{1}{|\mathbf{m}^n_e|} - \frac{1}{|\widetilde{\mathbf{m}}^n|} \Big) \Big) 
        + \nabla_h \frac{\widetilde{\mathbf{e}}^n}{|\widetilde{\mathbf{m}}^n|} \Big| \\
        \le & \Big| \nabla_h \mathbf{m}^n_e\cdot\Pi_h \Big(\frac{1}{|\mathbf{m}^n_e|} - \frac{1}{|\widetilde{\mathbf{m}}^n|} \Big) \Big| + \Big| \Pi_h \mathbf{m}^n_e \cdot \nabla_h \Big( \frac{1}{|\mathbf{m}^n_e|} - \frac{1}{|\widetilde{\mathbf{m}}^n|} \Big) \Big| \\
        & + \Big| \Pi_h\frac{1}{|\widetilde{\mathbf{m}}^n|}\cdot \nabla_h\widetilde{\mathbf{e}}^n \Big| 
        + \Big| \Pi_h\widetilde{\mathbf{e}}^n \cdot \nabla_h\frac{1}{|\widetilde{\mathbf{m}}^n|} \Big| 
        := J_1 + J_2 + J_3 + J_4.
        \endaligned
    \setlength{\belowdisplayskip}{4pt}
    \end{equation}
The terms $J_i$ ($i=1,\cdots,4$) will be separately analyzed. Motivated by an identity for any non-zero numbers $a$ and $b$, $\frac{1}{|a|} - \frac{1}{|b|} = -\frac{(a - b)(a + b)}{|a| \cdot |b|(|a| + |b|)}$, we immediately observe a bound for the first term: 
     \begin{equation}\label{lem4.3eq1}
         \setlength{\abovedisplayskip}{4pt}
         J_1 = \Big| \nabla_h \mathbf{m}^n_e \cdot \Pi_h \frac{\widetilde{\mathbf{e}}^n( \mathbf{m}^n_e + \widetilde{\mathbf{m}}^n)}{|\mathbf{m}^n_e| \cdot |\widetilde{\mathbf{m}}^n|(|\mathbf{m}^n_e| + |\widetilde{\mathbf{m}}^n|)} \Big|
         \leq |\nabla_h \mathbf{m}^n_e| \cdot \Pi_h|\widetilde{\mathbf{e}}^n | ,  \quad 
         \mbox{since} \, \, \, |\mathbf{m}^n_e| = 1, \, \, |\widetilde{\mathbf{m}}^n| \ge 1. 
         \setlength{\belowdisplayskip}{4pt}
     \end{equation}
The term of $J_2$ could be analyzed as follows: 
     \begin{align} 
         \setlength{\abovedisplayskip}{4pt}
         J_2 = & \Big| \Pi_h \mathbf{m}^n_{e,i+1/2}\cdot \Big( \frac{\nabla_h \mathbf{m}^n_{e,i+1/2}\cdot(\mathbf{m}^n_{e,i+1} + \mathbf{m}^n_{e,i})}{|\mathbf{m}^n_{e,i+1}| \cdot |\mathbf{m}^n_{e,i}|(|\mathbf{m}^n_{e,i+1}|+|\mathbf{m}^n_{e,i}|)} - \frac{\nabla_h \widetilde{\mathbf{m}}^n_{i+1/2}\cdot(\widetilde{\mathbf{m}}^n_{i+1} + \widetilde{\mathbf{m}}^n_{i})}{|\widetilde{\mathbf{m}}^n_{i+1}| \cdot |\widetilde{\mathbf{m}}^n_i|(|\widetilde{\mathbf{m}}^n_{i+1}|+|\widetilde{\mathbf{m}}^n_i|)} \Big) \Big| \notag\\
         = & \left|\Pi_h \mathbf{m}_{e,i+1/2}\cdot\frac{|\widetilde{\mathbf{m}}^n_{i+1}| \cdot |\widetilde{\mathbf{m}}^n_i|(|\widetilde{\mathbf{m}}^n_{i+1}|+|\widetilde{\mathbf{m}}^n_i|)\nabla_h \mathbf{m}^n_{e,i+1/2}\cdot(\mathbf{m}^n_{e,i+1} + \mathbf{m}^n_{e,i})}{|\mathbf{m}^n_{e,i+1}| \cdot |\mathbf{m}^n_{e,i}|(|\mathbf{m}^n_{e,i+1}|+|\mathbf{m}^n_{e,i}|)|\widetilde{\mathbf{m}}^n_{i+1}| \cdot |\widetilde{\mathbf{m}}^n_i|(|\widetilde{\mathbf{m}}^n_{i+1}|+|\widetilde{\mathbf{m}}^n_i|)}\right.\label{lem4.3eq3}\\
         & \left. -\Pi_h \mathbf{m}^n_{e,i+1/2}\cdot \frac{|\mathbf{m}^n_{e,i+1}| \cdot |\mathbf{m}^n_{e,i}|(|\mathbf{m}^n_{e,i+1}|+|\mathbf{m}^n_{e,i}|)\nabla_h \widetilde{\mathbf{m}}^n_{i+1/2}\cdot(\widetilde{\mathbf{m}}^n_{i+1} + \widetilde{\mathbf{m}}^n_{i})}{|\mathbf{m}^n_{e,i+1}| \cdot |\mathbf{m}^n_{e,i}|(|\mathbf{m}^n_{e,i+1}|+|\mathbf{m}^n_{e,i}|)|\widetilde{\mathbf{m}}^n_{i+1}| \cdot |\widetilde{\mathbf{m}}^n_i| (|\widetilde{\mathbf{m}}^n_{i+1}|+|\widetilde{\mathbf{m}}^n_i|)}\right|\notag\\
         \le & C ( |\Pi_h \mathbf{m}^n_e| \cdot |\nabla_h \mathbf{m}^n_e|\cdot\Pi_h|\widetilde{\mathbf{e}}^n| 
         + |\Pi_h \mathbf{m}^n_e| \cdot |\nabla_h\widetilde{\mathbf{e}}^n| ) 
         \le C (  |\nabla_h \mathbf{m}^n_e| \cdot \Pi_h|\widetilde{\mathbf{e}}^n| 
         + |\nabla_h\widetilde{\mathbf{e}}^n| ) , \notag 
     \setlength{\belowdisplayskip}{4pt}
     \end{align}
in which the fact that $| \mathbf{m}^n_e| \equiv 1$ has been used. A bound for the third term on the right hand of \eqref{all} is more straightforward: 
     \begin{equation}\label{lem4.3eq4}
         \setlength{\abovedisplayskip}{4pt}
         J_3 \leq |\nabla_h \widetilde{\mathbf{e}}^n| , \quad \mbox{since} \, \, \, 
         |\widetilde{\mathbf{m}}^n| \ge 1. 
        \setlength{\belowdisplayskip}{4pt}
     \end{equation}
The last term on the right hand of \eqref{all} could be bounded by
     \begin{equation}\label{lem4.3eq2}
         \setlength{\abovedisplayskip}{4pt}
         \begin{aligned} 
         J_4 = & \Big| \Pi_h\widetilde{\mathbf{e}}_{i+1/2}^n\cdot \frac{1}{h} \Big(\frac{1}{|\widetilde{\mathbf{m}}^n_{i+1}|} - \frac{1}{|\widetilde{\mathbf{m}}^n_{i}|} \Big) \Big| = \Big| \Pi_h\widetilde{\mathbf{e}}^n_{i+1/2} \cdot \frac{\nabla_h\widetilde{\mathbf{m}}^n_{i+1/2}\cdot(\widetilde{\mathbf{m}}^n_{i+1} + \widetilde{\mathbf{m}}^n_{i})}{|\widetilde{\mathbf{m}}^n_{i+1}| \cdot |\widetilde{\mathbf{m}}^n_i|(|\widetilde{\mathbf{m}}^n_{i+1}| + |\widetilde{\mathbf{m}}^n_i|)} \Big| \\
         \le & \Pi_h|\widetilde{\mathbf{e}}^n|\cdot|\nabla_h \widetilde{\mathbf{m}}^n| 
         \le \Pi_h|\widetilde{\mathbf{e}}^n|\cdot|\nabla_h \widetilde{\mathbf{e}}^n| 
         + \Pi_h|\widetilde{\mathbf{e}}^n|\cdot|\nabla_h \mathbf{m}^n_e|. 
        \end{aligned} 
         \setlength{\belowdisplayskip}{4pt}
     \end{equation}
Finally, with the help of discrete H\"{o}lder inequality, a substitution of \eqref{lem4.3eq1}-\eqref{lem4.3eq2} into \eqref{all} leads to 
     \begin{equation}
        \setlength{\abovedisplayskip}{4pt}
         \|\nabla_h \mathbf{e}^n\|^2_{2}\leq C((1 + \|\widetilde{\mathbf{e}}^n\|_{\infty}^2)\|\nabla_h \widetilde{\mathbf{e}}^n\|^2_{2} + \|\nabla_h\mathbf{m}^n_e\|^2_\infty \cdot \|\widetilde{\mathbf{e}}^n\|^2_2 ) 
         \le M_0^2 (\|\widetilde{\mathbf{e}}^n\|^2_{2} + \|\nabla_h \widetilde{\mathbf{e}}^n\|^2_{2}),
        \setlength{\belowdisplayskip}{4pt}
     \end{equation}
with $M_0$ only dependent on $\Omega$ and \(\|\nabla_h\mathbf{m}^n_e\|^2_\infty\). This completes the proof of Lemma~\ref{lem6}. 
    \end{proof}
    \medskip
    \begin{remark}
In the existing works of An et al. \cite{an2021optimal} and Gui et al. \cite{gui2022convergence}, an estimate between the intermediate numerical error \(\widetilde{\mathbf{e}}^n\) and the projected numerical error \(\mathbf{e}^n\) is established as 
        \begin{equation} 
          \setlength{\abovedisplayskip}{4pt}
            \|\mathbf{e}^n\|^2_2 \leq \|\widetilde{\mathbf{e}}^n\|^2_2 + \text{higher-order terms}. 
            \label{error projection-existing-1} 
          \setlength{\belowdisplayskip}{4pt}
        \end{equation}
However, these higher order terms are highly nonlinear, and a precise control of them turns out to be a challenging issue. In comparison, the projected error estimate \eqref{e_l^2} in our work is much sharper, which comes from a point-wise lower bound \(|\widetilde{\mathbf{m}}^n| \geq 1\) for the constructed scheme \eqref{discretescheme}. 
    \end{remark}

Now we present the main result in the following theorem for the fully discrete scheme \eqref{discretescheme}.
\medskip
\begin{theorem}\label{thm1error}
    Assume that the exact solution $\mathbf{m}_e(\mathbf{x},t)$ to the LLG equation \eqref{e_original model} satisfies $\mathbf{m}_e(\mathbf{x},t) \in H^2(0,\,T;\, C (\Omega)) \bigcap \textcolor{black}{C^1 (0,\,T;\, \mathbf{H}^2 (\Omega))} \bigcap L^\infty (0,\,T;\,\mathbf{H}^4(\Omega))$. For the fully discrete scheme \eqref{discretescheme}, if there exists a positive constant $0 <\epsilon_0<1$, such that
    \begin{equation}\label{CFLcondition}
    \setlength{\abovedisplayskip}{4pt}
    \left\{
    \begin{array}{ll}
        \textcolor{black}{h^2\lesssim \Delta t \lesssim h^{\epsilon_0}},&\text{in the 2D case},  \\
         h^2\lesssim \Delta t \lesssim h^{1+\epsilon_0},& \text{in the 3D case},
    \end{array} \right.  
    \setlength{\belowdisplayskip}{4pt}
    \end{equation}
then an optimal rate error estimate is available: 
    \begin{equation} 
    \setlength{\abovedisplayskip}{4pt}
        \|\mathbf{e}^n\|^2_{2} + \Delta t \sum^{n}_{k=1}\|\nabla_h \mathbf{e}^{k}\|^2_{2} \leq \hat{C}_0(h^4 + \Delta t^2) ,  \quad \mbox{for} \, \, 1 \le n \le N ,  \label{convergence-0} 
    \setlength{\belowdisplayskip}{4pt}
    \end{equation}
provided that $\Delta t$ and $h$ are sufficiently small, where $\hat{C}_0 > 0$ is independent of $h$ and $\Delta t$.
\end{theorem}

\begin{proof}
The following discrete bounds are available for the exact solution $\mathbf{m}_e$, due to its regularity: 
\begin{equation}
 \setlength{\abovedisplayskip}{3pt}
 \| \nabla_h \mathbf{m}_e^k \|_{p_0} , \, \,  \Big\| \frac{ \mathbf{m}_e^{k+1} - \mathbf{m}_e^k }{\Delta t} \Big\|_{H^1_h} \le C^* , \quad \forall k \ge 0   
 \setlength{\belowdisplayskip}{3pt}
  \label{exact-inf-1} 
\end{equation}
and we make the following a-priori assumption for the numerical error function at the previous time step:
\begin{equation}\label{a priori-1} 
  \setlength{\abovedisplayskip}{3pt} 
\| \widetilde{\mathbf{e}}^{n-1} \|_2  
  \le \Delta t^{1 - \frac{\epsilon_0}{8}} + h^{2 - \frac{\epsilon_0}{4}} , \quad   
 \| \nabla_h \widetilde{\mathbf{e}}^{n-1} \|_2  \le \Delta t^{\frac12 - \frac{\epsilon_0}{8}} 
   + h^{1 - \frac{\epsilon_0}{4}} . 
   \setlength{\belowdisplayskip}{3pt}
\end{equation}
Such an a-prior assumption is valid at $n=0$, and it will be recovered by the optimal rate convergence analysis at the next time step, as will be proved later. In turn, a discrete $\ell^\infty$ bound for the numerical error function at the intermediate stage becomes available: 
\begin{equation}\label{a priori-2} 
  \setlength{\abovedisplayskip}{3pt}
\begin{aligned} 
   \| \widetilde{\mathbf{e}}^{n-1} \|_\infty \le \frac{C ( \| \widetilde{\mathbf{e}}^{n-1} \|_2 
   + \| \nabla_h \widetilde{\mathbf{e}}^{n-1} \|_2 ) }{h^\frac12} \le C ( \Delta t^{\frac12 - \frac{\epsilon_0}{8}} h^{-\frac12} + h^{\frac12 - \frac{\epsilon_0}{4}} )   \le C_1 h^{\frac{\epsilon_0}{3}} \le 1 , 
\end{aligned} 
  \setlength{\belowdisplayskip}{3pt}
\end{equation} 
in which the scaling law that {\color{black} $\Delta t^{\frac12 - \frac{\epsilon_0}{8}} \lesssim h^{(1 + \epsilon_0) ( \frac12 - \frac{\epsilon_0}{8} )} \lesssim h^{\frac12 + \frac{\epsilon_0}{3}}$} has been applied, provided that $\epsilon_0$ is sufficiently small. 
Under such an error bound, an application of the preliminary estimate~\eqref{e_lh1} gives 
\begin{equation} 
  \setlength{\abovedisplayskip}{4pt} 
\begin{aligned} 
  \|\nabla_h \mathbf{e}^{n-1} \|_2 \le & M_0 (\|\widetilde{\mathbf{e}}^{n-1} \|_2 
            + \|\nabla_h \widetilde{\mathbf{e}}^{n-1} \|_2 ) 
            \le (M_0 + 1/2) ( \Delta t^{\frac12 - \frac{\epsilon_0}{8}}   
   + h^{1 - \frac{\epsilon_0}{4}} ) 
\\
   \le & 
     (M_0 + 1/2) ( h^{\frac12 + \frac{\epsilon_0}{3}}  
   + h^{1 - \frac{\epsilon_0}{4}} ) \le (M_0 + 1)  h^{\frac12 + \frac{\epsilon_0}{3}} . 
\end{aligned} 
  \setlength{\belowdisplayskip}{4pt} 
   \label{a priori-3}
\end{equation} 
Based on the fact that $\frac32 - \frac{3}{p_0} = \frac12 + \frac{\epsilon_0}{4}$, an application of inverse inequality implies that 
\begin{equation} 
  \setlength{\abovedisplayskip}{4pt}  
  \|\nabla_h \mathbf{e}^{n-1} \|_{p_0} \le C h^{-(\frac12 +  \frac{\epsilon_0}{4})} 
  \cdot \|\nabla_h \mathbf{e}^{n-1} \|_2
  \le C_2 (M_0 + 1)  h^{\epsilon_0 /12} \le 1/2 . 
  \setlength{\belowdisplayskip}{4pt} 
   \label{a priori-4}
\end{equation} 
  Subsequently, 
  a discrete $W_h^{1, p_0}$ bound for the numerical solution is valid at the previous time step: 
\begin{equation} \label{a priori-5}
  \setlength{\abovedisplayskip}{4pt}  
    \|\nabla_h \mathbf{m}^{n-1} \|_{p_0} \le \|\nabla_h \mathbf{m}_e^{n-1} \|_{p_0}  
    + \|\nabla_h \mathbf{e}^{n-1} \|_{p_0} \le C^* + 1/2 = \tilde{C}_1. 
  \setlength{\belowdisplayskip}{4pt} 
\end{equation}
Next, we proceed with the detailed error estimates. 
Taking $\mathbf{v} = \widetilde{\mathbf{e}}^n$ in \eqref{merror} gives 
    \begin{equation}\label{errorall}
    \setlength{\abovedisplayskip}{4pt}  
        \frac{1}{2 \Delta t} ( \|\widetilde{\mathbf{e}}^n\|_2^2 - \|\mathbf{e}^{n-1}\|_2^2 + \|\widetilde{\mathbf{e}}^n - \mathbf{e}^{n-1}\|^2_2 ) + \gamma \|\nabla_h\widetilde{\mathbf{e}}^n\|^2_2 
        = \sum^9_{i=1}I_i + \langle \mathbf{R}^n, \widetilde{\mathbf{e}}^n \rangle . 
    \setlength{\belowdisplayskip}{4pt}    
    \end{equation}
The first term on the right-hand side of \eqref{errorall} vanishes, since the two vectors are orthogonal: 
    \begin{equation}\label{I_1} 
     \setlength{\abovedisplayskip}{4pt} 
        I_1 = \beta \langle \nabla_h \widetilde{\mathbf{e}}^n, 
        \nabla_h\widetilde{\mathbf{e}}^n\times \Pi_h\mathbf{m}^{n-1} \rangle = 0. 
      \setlength{\belowdisplayskip}{4pt} 
    \end{equation}
With the help of the a-priori bound \eqref{a priori-5} for the numerical solution, an application of discrete H\"{o}lder inequality gives a bound for the second term:  
    \begin{equation}\label{I_2}
     \setlength{\abovedisplayskip}{4pt} 
        I_2 = \beta \langle \nabla_h \widetilde{\mathbf{e}}^n,\Pi_h\widetilde{\mathbf{e}}^n \times \nabla_h \mathbf{m}^{n-1} \rangle \leq \beta \|\nabla_h\widetilde{\mathbf{e}}^n\|_2 \cdot \|\widetilde{\mathbf{e}}^n\|_{q_0} \cdot \|\nabla_h\mathbf{m}^{n-1}\|_{p_0} \leq \beta\tilde{C}_1\|\nabla_h\widetilde{\mathbf{e}}^n\|_2 \cdot \|\widetilde{\mathbf{e}}^n\|_{q_0}. 
     \setlength{\belowdisplayskip}{4pt} 
    \end{equation}
The third and forth terms could be controlled with the help of the bound~\eqref{exact-inf-1} for the exact solution: 
    \begin{align}\label{I_34}
        \setlength{\abovedisplayskip}{4pt} 
        I_3 = \beta \langle \nabla_h \mathbf{m}^n_e,\nabla_h \widetilde{\mathbf{e}}^n \times \Pi_h\mathbf{e}^{n-1} \rangle \leq \beta \|\nabla_h\widetilde{\mathbf{e}}^n\|_2 \cdot \|\nabla_h \mathbf{m}^n_e \cdot \|_{p_0}\|\mathbf{e}^{n-1}\|_{q_0} \leq \beta C^*\|\nabla_h\widetilde{\mathbf{e}}^n\|_2 \cdot \|\mathbf{e}^{n-1}\|_{q_0},\\
        I_4 = \beta \langle \nabla_h \mathbf{m}_e^n, \Pi_h\widetilde{\mathbf{e}}^n \times \nabla_h \mathbf{e}^{n-1} \rangle \leq \beta\|\nabla_h \mathbf{m}^n_e\|_{p_0} \cdot \|\widetilde{\mathbf{e}}^{n}\|_{q_0} \cdot \|\nabla_h\mathbf{e}^{n-1}\|_2\leq \beta C^*\|\widetilde{\mathbf{e}}^{n}\|_{q_0} \cdot \|\nabla_h\mathbf{e}^{n-1}\|_2. 
     \setlength{\belowdisplayskip}{4pt} 
    \end{align}
Regarding the fifth and sixth terms on the right-hand side of \eqref{errorall}, we see that 
    \begin{align}
     \setlength{\abovedisplayskip}{4pt} 
        I_5 &= \gamma  \langle \nabla_h \mathbf{m}_e^n : \nabla_h \mathbf{m}_e^{n-1} , \Pi_h(\mathbf{e}^{n-1} \cdot\widetilde{\mathbf{e}}^n )  \rangle \notag\\
        &\leq \gamma\|\nabla_h \mathbf{m}^n_e\|_{p_0}\|\nabla_h\mathbf{m}^{n-1}\|_{p_0}\|\mathbf{e}^{n-1}\|_{q_0}\|\widetilde{\mathbf{e}}^n\|_{q_0} \leq \gamma(C^*)^2\|\mathbf{e}^{n-1}\|_{q_0}\|\widetilde{\mathbf{e}}^n\|_{q_0},\label{I_5}\\
        I_6 &= \gamma \langle \nabla_h\widetilde{\mathbf{e}}^{n}: \nabla_h\mathbf{m}^{n-1} 
    + \nabla_h \mathbf{m}_e^{n}: \nabla_h\mathbf{e}^{n-1} , 
    \Pi_h(\mathbf{m}^{n-1} \cdot\widetilde{\mathbf{e}}^n ) \rangle \notag\\
    &\leq \gamma(\|\nabla_h\widetilde{\mathbf{e}}^n\|_{2} \cdot \|\nabla_h\mathbf{m}^{n-1}\|_{p_0} 
    + \|\nabla_h\mathbf{m}^n_e\|_{p_0} \cdot \|\nabla_h \mathbf{e}^{n-1}\|_{2}) 
    \|\mathbf{m}^{n-1}\|_{\infty} \cdot \|\widetilde{\mathbf{e}}^n\|_{q_0}\\
    &\leq\gamma(\tilde{C}_1\|\nabla_h\widetilde{\mathbf{e}}^{n}\|_{2} + C^*\|\nabla_h\mathbf{e}^{n-1}\|_{2})\|\widetilde{\mathbf{e}}^n\|_{q_0}.\notag
    \setlength{\belowdisplayskip}{4pt} 
    \end{align}
Similarly, with the help of \eqref{exact-inf-1} and \eqref{a priori-5}, the seventh and eighth terms can be bounded by
    \begin{align}
    \setlength{\abovedisplayskip}{4pt} 
        I_7 = & - \gamma \langle ( \widetilde{\mathbf{e}}^n - \mathbf{e}^{n-1} ) \cdot \nabla_h \mathbf{m}^{n-1} , \nabla_h ( \mathbf{m}^{n-1} \cdot \widetilde{\mathbf{e}}^n ) \rangle \notag\\
        \le & \gamma\|\widetilde{\mathbf{e}}^n - \mathbf{e}^{n-1}\|_{q_0} \cdot \|\nabla_h\mathbf{m}^{n-1}\|_{p_0}(\|\nabla_h\mathbf{m}^{n-1}\|_{p_0} \cdot \|\widetilde{\mathbf{e}}^n\|_{q_0} + \|\mathbf{m}^{n-1}\|_{\infty} \cdot \|\nabla_h\widetilde{\mathbf{e}}^{n}\|_{2})\\
        \le &  \gamma \widetilde{C}_1(\|\widetilde{\mathbf{e}}^n\|_{q_0} + \|\mathbf{e}^{n-1}\|_{q_0})(\widetilde{C}_1\|\widetilde{\mathbf{e}}^n\|_{q_0} +\|\nabla_h\widetilde{\mathbf{e}}^n\|_2),\notag\\
        I_8 = & - \gamma \langle ( \mathbf{m}_e^n - \mathbf{m}_e^{n-1} ) \cdot \nabla_h \mathbf{e}^{n-1} , 
     \nabla_h ( \mathbf{m}_e^{n-1} \cdot \widetilde{\mathbf{e}}^n ) \rangle \notag 
\\
  \le & 
  \gamma \| \mathbf{m}_e^n - \mathbf{m}_e^{n-1} \|_\infty \cdot \| \nabla_h \mathbf{e}^{n-1} \|_2  
     ( \| \nabla_h \mathbf{m}_e^{n-1} \|_\infty \cdot \| \widetilde{\mathbf{e}}^n \|_2 
     + \| \mathbf{m}_e^{n-1} \|_\infty \cdot \| \nabla_h \widetilde{\mathbf{e}}^n \|_2  )  
     \label{convergence-NLE-8} 
\\
  \le & 
  \gamma ( C^*)^2 \Delta t  \| \nabla_h \mathbf{e}^{n-1} \|_2  
     (  \| \widetilde{\mathbf{e}}^n \|_2 + \| \nabla_h \widetilde{\mathbf{e}}^n \|_2  )  
     \le  \gamma \| \nabla_h \mathbf{e}^{n-1} \|_2  
     (  \| \widetilde{\mathbf{e}}^n \|_2 + \| \nabla_h \widetilde{\mathbf{e}}^n \|_2  ) / (16 M_0) . \notag
    \setlength{\belowdisplayskip}{4pt} 
    \end{align}
Moreover, an application of Cauchy-Schwarz inequality to the last two terms results in 
    \begin{align}
    \setlength{\abovedisplayskip}{4pt} 
       I_9 = & - \gamma \langle ( \mathbf{m}_e^n - \mathbf{m}_e^{n-1} ) \cdot \nabla_h \mathbf{m}^{n-1} , 
  \nabla_h ( \mathbf{e}^{n-1} \cdot \widetilde{\mathbf{e}}^n ) \rangle  \notag 
\\
  \le & \gamma \| \mathbf{m}_e^n - \mathbf{m}_e^{n-1} \|_\infty 
  \cdot  \| \nabla_h \mathbf{m}^{n-1} \|_{p_0}  
  \cdot ( \| \nabla_h \mathbf{e}^{n-1} \|_2 \cdot \| \widetilde{\mathbf{e}}^n  \|_{q_0} 
  + \| \mathbf{e}^{n-1} \|_{q_0} \cdot \| \nabla_h \widetilde{\mathbf{e}}^n  \|_2 )  \notag 
\\
  \le & 
  \gamma C^* \tilde{C}_1 \Delta t 
   ( \| \nabla_h \mathbf{e}^{n-1} \|_2 \cdot \| \widetilde{\mathbf{e}}^n  \|_{q_0} 
  + \| \mathbf{e}^{n-1} \|_{q_0} \cdot \| \nabla_h \widetilde{\mathbf{e}}^n  \|_2 ) , 
        \label{I9} \\
        \langle \mathbf{R}^n, \widetilde{\mathbf{e}}^n \rangle 
    \le & \| \mathbf{R}^n \|_2 \cdot \| \widetilde{\mathbf{e}}^n \|_2 
    \le \frac12 ( \| \mathbf{R}^n \|_2^2 + \| \widetilde{\mathbf{e}}^n \|_2^2 ) ,  \label{R}  
    \setlength{\belowdisplayskip}{4pt} 
    \end{align}
in which the bound~\eqref{exact-inf-1} has been applied in the derivation of~\eqref{I9}. In turn, a substitution of \eqref{I_1}-\eqref{R} into \eqref{errorall} leads to 
\begin{equation} 
  \setlength{\abovedisplayskip}{4pt}    
  \begin{aligned} 
    & 
    \frac{1}{2 \Delta t} ( \|\widetilde{\mathbf{e}}^n\|_2^2 - \|\mathbf{e}^{n-1}\|_2^2 + \|\widetilde{\mathbf{e}}^n - \mathbf{e}^{n-1}\|^2_2 ) + \gamma \|\nabla_h\widetilde{\mathbf{e}}^n\|^2_2 
\\
  \le & 
   \tilde{C}_2  \| \mathbf{e}^{n-1} \|_2^2 + ( \tilde{C}_2  + \frac12 ) \| \widetilde{\mathbf{e}}^n \|_2^2   
  + \tilde{C}_3 \Big( \frac32 \| \widetilde{\mathbf{e}}^n \|_{q_0}^2  
  + \frac12 \| \mathbf{e}^{n-1} \|_{q_0}^2 \Big)    
  + \frac{\gamma}{16 M_0} \| \nabla_h \widetilde{\mathbf{e}}^n \|_2  
  \cdot \| \nabla_h \mathbf{e}^{n-1} \|_2   
\\
  & 
  + \| \nabla_h \mathbf{e}^{n-1} \|_2 ( ( \beta C^* + 1) \| \widetilde{\mathbf{e}}^n \|_2 
  + \tilde{C}_4 \| \widetilde{\mathbf{e}}^n \|_{q_0}  + \tilde{C}_3 \| \mathbf{e}^{n-1} \|_{q_0} ) 
\\
 & 
  + \| \nabla_h \widetilde{\mathbf{e}}^n \|_2 ( \beta C^* \| \mathbf{e}^{n-1} \|_2 
  + \tilde{C}_5 \| \widetilde{\mathbf{e}}^n \|_{q_0}  + ( 1 + \tilde{C}_3 ) \| \mathbf{e}^{n-1} \|_{q_0} ) 
   + \frac12  \| \mathbf{R}^n \|_2^2  , 
  \setlength{\belowdisplayskip}{4pt}  
\end{aligned} 
    \label{convergence-3-2} 
\end{equation} 
with $\tilde{C}_2 = \frac{\gamma (C^*)^2}{2}$, $\tilde{C}_3 = \gamma \tilde{C}_1 (\tilde{C}_1 + 1)$, $\tilde{C}_4 = \gamma \tilde{C}_1 + \tilde{C}_3$, $\tilde{C}_5 = (\beta + \gamma ) \tilde{C}_1 + \tilde{C}_3$. Moreover, a repeated application of Cauchy inequality gives 
\begin{align} 
  \setlength{\abovedisplayskip}{4pt}  
    & 
   \frac{\gamma}{16 M_0} \| \nabla_h \widetilde{\mathbf{e}}^n \|_2  
  \cdot \| \nabla_h \mathbf{e}^{n-1} \|_2 
  \le \frac{\gamma}{16} \| \nabla_h \widetilde{\mathbf{e}}^n \|_2^2  
  + \frac{\gamma}{64 M_0^2} \| \nabla_h \mathbf{e}^{n-1} \|_2^2 , 
    \label{convergence-3-3} 
\\
  & 
  \| \nabla_h \mathbf{e}^{n-1} \|_2 ( ( \beta C^* + 1) \| \widetilde{\mathbf{e}}^n \|_2 
  + \tilde{C}_4 \| \widetilde{\mathbf{e}}^n \|_{q_0}  + \tilde{C}_3 \| \mathbf{e}^{n-1} \|_{q_0} )  \notag 
\\
  \le & 
  \frac{\gamma}{64 M_0^2} \| \nabla_h \mathbf{e}^{n-1} \|_2^2 
  + 16 M_0^2 \gamma^{-1}( ( \beta C^* + 1) \| \widetilde{\mathbf{e}}^n \|_2 
  + \tilde{C}_4 \| \widetilde{\mathbf{e}}^n \|_{q_0}  + \tilde{C}_3 \| \mathbf{e}^{n-1} \|_{q_0} )^2  \notag 
\\
  \le & 
  \frac{\gamma}{64 M_0^2} \| \nabla_h \mathbf{e}^{n-1} \|_2^2 
  + 48 M_0^2 \gamma^{-1}( ( \beta C^* + 1)^2 \| \widetilde{\mathbf{e}}^n \|_2^2 
  + \tilde{C}_4^2 \| \widetilde{\mathbf{e}}^n \|_{q_0}^2  + \tilde{C}_3^2 \| \mathbf{e}^{n-1} \|_{q_0}^2 ) , 
   \label{convergence-3-4} 
\\
  & 
  \| \nabla_h \widetilde{\mathbf{e}}^n \|_2 ( \beta C^* \| \mathbf{e}^{n-1} \|_2 
  + \tilde{C}_5 \| \widetilde{\mathbf{e}}^n \|_{q_0}  + ( 1 + \tilde{C}_3 ) \| \mathbf{e}^{n-1} \|_{q_0} ) \notag 
\\
  \le & 
  \frac{\gamma}{16} \| \nabla_h \widetilde{\mathbf{e}}^n \|_2^2 
  + 12 \gamma^{-1}  ( \beta^2 ( C^*)^2 \| \mathbf{e}^{n-1} \|_2^2  
  + \tilde{C}_5^2 \| \widetilde{\mathbf{e}}^n \|_{q_0}^2  
  + ( 1 + \tilde{C}_3 )^2 \| \mathbf{e}^{n-1} \|_{q_0}^2 )  .  
  \label{convergence-3-5} 
  \setlength{\belowdisplayskip}{4pt}  
\end{align} 
 Going back \eqref{convergence-3-2}, we get  
\begin{equation} 
  \setlength{\abovedisplayskip}{4pt}    
  \begin{aligned} 
    & 
    \frac{1}{2 \Delta t} ( \|\widetilde{\mathbf{e}}^n\|_2^2 - \|\mathbf{e}^{n-1}\|_2^2 + \|\widetilde{\mathbf{e}}^n - \mathbf{e}^{n-1}\|^2_2 ) + \frac{7 \gamma}{8} \|\nabla_h\widetilde{\mathbf{e}}^n \|^2_2 
    - \frac{\gamma}{32 M_0^2} \|\nabla_h \mathbf{e}^{n-1} \|^2_2 
\\
  \le & 
  \tilde{C}_5  \| \mathbf{e}^{n-1} \|_2^2 + \tilde{C}_6 \| \widetilde{\mathbf{e}}^n \|_2^2   
  + \tilde{C}_7  \| \widetilde{\mathbf{e}}^n \|_{q_0}^2  
  + \tilde{C}_8 \| \mathbf{e}^{n-1} \|_{q_0}^2 + \frac12  \| \mathbf{R}^n \|_2^2 , 
  \setlength{\belowdisplayskip}{4pt}  
\end{aligned} 
    \label{convergence-3-6} 
\end{equation}   
with $\tilde{C}_5 = \tilde{C}_2 + 12 \gamma^{-1}  \beta^2 ( C^*)^2$, $\tilde{C}_6 = \tilde{C}_2 + \frac12 + 48 M_0^2 \gamma^{-1} ( \beta C^* + 1)^2$, $\tilde{C}_7 = \frac32 \tilde{C}_3 + 48 M_0^2 \gamma^{-1} \tilde{C}_4^2 + 12 \gamma^{-1} \tilde{C}_5^2$, $\tilde{C}_8 = \frac12 \tilde{C}_3 + 48 M_0^2 \gamma^{-1} \tilde{C}_3^2 + 12 \gamma^{-1} ( 1+ \tilde{C}_3 )^2$. Meanwhile, with $2 < q_0 = \frac{6}{1 + \frac12 \epsilon_0} < 6$, an application of the interpolation inequality~\eqref{interpolation-1} implies that 
\begin{equation} 
    \setlength{\abovedisplayskip}{4pt}  
\begin{aligned} 
  & 
  \|\boldsymbol{f}\|_{q_0} \le  
   \breve{C}_0 \|\boldsymbol{f}\|_2^{\frac{\epsilon_0}{4}} \cdot 
   ( \| \boldsymbol{f} \|_2 + \| \nabla_h \boldsymbol{f} \|_2 )^{1 - \frac{\epsilon_0}{4} }   
   \le \breve{C}_1 ( \| \boldsymbol{f} \|_2 + \| \boldsymbol{f} \|_2^{\frac{\epsilon_0}{4}} \cdot 
    \| \nabla_h \boldsymbol{f} \|_2^{1 - \frac{\epsilon_0}{4} }  )  , 
\\
  & \mbox{so that}  \quad 
   \|\boldsymbol{f}\|_{q_0}^2  
   \le 2 \breve{C}_1^2 ( \| \boldsymbol{f} \|_2^2  
   + \| \boldsymbol{f} \|_2^{\frac{\epsilon_0}{2}} \cdot 
    \| \nabla_h \boldsymbol{f} \|_2^{2 - \frac{\epsilon_0}{2} }  )  
    \le C ( \epsilon_0, \alpha )  \| \boldsymbol{f} \|_2^2  
    + \alpha \| \nabla_h \boldsymbol{f} \|_2^2,  \quad \forall \alpha > 0 , 
\\
  & 
  \tilde{C}_7  \| \widetilde{\mathbf{e}}^n \|_{q_0}^2  
  \le \tilde{C}_9  (\epsilon_0, \gamma) \| \widetilde{\mathbf{e}}^n \|_2^2 
  + \frac{\gamma}{8} \| \nabla_h \widetilde{\mathbf{e}}^n \|_2^2 , 
\\
  &
   \tilde{C}_8  \| \mathbf{e}^{n-1} \|_{q_0}^2  
  \le \tilde{C}_{10}  (\epsilon_0, \gamma, M_0) \| \mathbf{e}^{n-1} \|_2^2 
  + \frac{\gamma}{32 M_0^2} \| \nabla_h \mathbf{e}^{n-1} \|_2^2 ,   
   \setlength{\belowdisplayskip}{4pt}  
\end{aligned}  
    \label{convergence-3-7} 
\end{equation}  
in which the Young's inequality has been repeatedly applied. Its substitution into~\eqref{convergence-3-6} leads to 
\begin{equation} 
  \setlength{\abovedisplayskip}{4pt}    
  \begin{aligned} 
    & 
    \frac{1}{2 \Delta t} ( \|\widetilde{\mathbf{e}}^n\|_2^2 - \|\mathbf{e}^{n-1}\|_2^2 + \|\widetilde{\mathbf{e}}^n - \mathbf{e}^{n-1}\|^2_2 ) + \frac{3 \gamma}{4} \|\nabla_h\widetilde{\mathbf{e}}^n \|^2_2 
    - \frac{\gamma}{16 M_0^2} \|\nabla_h \mathbf{e}^{n-1} \|^2_2 
\\
  \le & 
  \tilde{C}_{11}  \| \mathbf{e}^{n-1} \|_2^2 + \tilde{C}_{12} \| \widetilde{\mathbf{e}}^n \|_2^2   
  + \frac12  \| \mathbf{R}^n \|_2^2 ,    \quad  \tilde{C}_{11} = \tilde{C}_5 
  + \tilde{C}_{10}  (\epsilon_0, \gamma, M_0) , \, \, \,  
  \tilde{C}_{12} = \tilde{C}_6 + \tilde{C}_9  (\epsilon_0, \gamma) . 
  \setlength{\belowdisplayskip}{4pt}  
\end{aligned} 
    \label{convergence-3-10} 
\end{equation}  
Meanwhile, by the a-priori error estimate~\eqref{e_lh1}, we observe that $\|\nabla_h \mathbf{e}^{n-1} \|_2^2 \le 2 M_0^2  (\| \widetilde{\mathbf{e}}^{n-1} \|_2^2 +  \| \nabla_h \widetilde{\mathbf{e}}^{n-1} \|_2^2)$, and arrive at 
\begin{equation} 
  \setlength{\abovedisplayskip}{4pt}    
  \begin{aligned} 
    & 
    \frac{1}{2 \Delta t} ( \|\widetilde{\mathbf{e}}^n\|_2^2 - \|\mathbf{e}^{n-1}\|_2^2 + \|\widetilde{\mathbf{e}}^n - \mathbf{e}^{n-1}\|^2_2 ) + \frac{3 \gamma}{4} \|\nabla_h\widetilde{\mathbf{e}}^n \|^2_2 
    - \frac{\gamma}{8} \|\nabla_h\widetilde{\mathbf{e}}^{n-1} \|^2_2 
\\
  \le & 
  \tilde{C}_{11}  \| \mathbf{e}^{n-1} \|_2^2 + \tilde{C}_{12} \| \widetilde{\mathbf{e}}^n \|_2^2   
  + \frac{\gamma}{8}  \| \widetilde{\mathbf{e}}^{n-1} \|_2^2  + \frac12  \| \mathbf{R}^n \|_2^2 . 
  \setlength{\belowdisplayskip}{4pt}  
\end{aligned} 
    \label{convergence-3-11} 
\end{equation}   

Moreover, by the $\ell^2$ error estimate~\eqref{e_l^2} in the normalization stage, we get 
\begin{equation} 
  \setlength{\abovedisplayskip}{4pt}    
  \begin{aligned} 
    & 
    \frac{1}{2 \Delta t} ( \| \mathbf{e}^n\|_2^2 - \|\mathbf{e}^{n-1}\|_2^2 + \|\widetilde{\mathbf{e}}^n - \mathbf{e}^{n-1}\|^2_2 + \| \widetilde{\mathbf{e}}^n - \mathbf{e}^n \|_2^2 ) + \frac{3 \gamma}{4} \|\nabla_h\widetilde{\mathbf{e}}^n \|^2_2 
    - \frac{\gamma}{8} \|\nabla_h\widetilde{\mathbf{e}}^{n-1} \|^2_2 
\\
  \le & 
  \tilde{C}_{11}  \| \mathbf{e}^{n-1} \|_2^2 + \tilde{C}_{12} \| \widetilde{\mathbf{e}}^n \|_2^2   
  + \frac{\gamma}{8}  \| \widetilde{\mathbf{e}}^{n-1} \|_2^2  + \frac12  \| \mathbf{R}^n \|_2^2  
\\
  \le & 
  ( \tilde{C}_{11}  + \frac{\gamma}{4} ) \| \mathbf{e}^{n-1} \|_2^2 
  + 2 \tilde{C}_{12} \| \mathbf{e}^n \|_2^2   
  + 2 \tilde{C}_{12} \| \widetilde{\mathbf{e}}^n - \mathbf{e}^n \|_2^2 
  + \frac{\gamma}{4} \| \widetilde{\mathbf{e}}^{n-1} - \mathbf{e}^{n-1} \|_2^2
  + \frac12  \| \mathbf{R}^n \|_2^2  . 
  \setlength{\belowdisplayskip}{4pt}  
\end{aligned} 
    \label{convergence-3-12} 
\end{equation}   
A summation in time implies that 
\begin{equation} 
  \setlength{\abovedisplayskip}{4pt}    
    \| \mathbf{e}^n\|_2^2 
    + \frac12 \sum_{k=0}^n \| \widetilde{\mathbf{e}}^k - \mathbf{e}^k \|_2^2  
    + \frac{5 \gamma \Delta t}{4} \sum_{k=0}^n \|\nabla_h\widetilde{\mathbf{e}}^k \|^2_2 
  \le  
   \tilde{C}_{13}  \Delta t \sum_{k=0}^n \| \mathbf{e}^k \|_2^2  
  +  \Delta t \sum_{k=0}^n \| \mathbf{R}^k \|_2^2  ,  
  \setlength{\belowdisplayskip}{4pt}  
    \label{convergence-3-13} 
\end{equation}   
with $\tilde{C}_{13} = 2 \tilde{C}_{11}  + \frac{\gamma}{2} + 4 \tilde{C}_{12}$, provided that $2 \tilde{C}_{12} + \frac{\gamma}{4}  \le \frac{1}{4 \Delta t}$. Therefore, an application of discrete Gronwall inequality results in an optimal rate error analysis 
\begin{equation} 
  \setlength{\abovedisplayskip}{4pt}    
    \| \mathbf{e}^n\|_2  
    + \Big( \frac12 \sum_{k=0}^n \| \widetilde{\mathbf{e}}^k - \mathbf{e}^k \|_2^2  \Big)^\frac12 
    + \Big( \frac{5 \gamma \Delta t}{4} \sum_{k=0}^n \|\nabla_h\widetilde{\mathbf{e}}^k \|^2_2 \Big)^\frac12 
  \le  \hat{C} ( \Delta t+ h^2 ) . 
  \setlength{\belowdisplayskip}{4pt}  
    \label{convergence-3-14} 
\end{equation}  
which in turn yields the desired convergence estimate~\eqref{convergence-0}, by setting $\hat{C}_0 = M_0 \hat{C}$. Meanwhile, it is observed that the a-priori assumption~\eqref{a priori-1} is recovered at the next time step $t^n$: 
\begin{equation}\label{a priori-6} 
  \setlength{\abovedisplayskip}{4pt} 
  \begin{aligned} 
  & 
\| \mathbf{e}^n \|_2  \le  \hat{C} ( \Delta t+ h^2 ) \le \frac{1}{2}(\Delta t^{1 - \frac{\epsilon_0}{8}} 
 + h^{2 - \frac{\epsilon_0}{4}}) , \, 
  \| \widetilde{\mathbf{e}}^n - \mathbf{e}^n \|_2 \le \sqrt{2} \hat{C} ( \Delta t+ h^2 )
  \le \frac{1}{2}(\Delta t^{1 - \frac{\epsilon_0}{8}} + h^{2 - \frac{\epsilon_0}{4}}) , 
\\
  &\| \widetilde{\mathbf{e}
  }^n \|_2  \leq \| \mathbf{e}^n \|_2 +  \| \widetilde{\mathbf{e}}^n - \mathbf{e}^n \|_2 \leq \Delta t^{1 - \frac{\epsilon_0}{8}} 
 + h^{2 - \frac{\epsilon_0}{4}},\,
 \| \nabla_h \widetilde{\mathbf{e}}^n \|_2  \le \frac{\hat{C} ( \Delta t + h^2) }{\gamma^\frac12 \Delta t^\frac12} 
 \le \Delta t^{\frac12 - \frac{\epsilon_0}{8}} + h^{1 - \frac{\epsilon_0}{4}} , 
\end{aligned}  
   \setlength{\belowdisplayskip}{4pt}
\end{equation}
under the scaling law {\color{black} $h^2 \lesssim \Delta t \lesssim h^{1+\epsilon_0}$}, provided that $\Delta t$ and $h$ are sufficiently small. This completes the proof of Theorem~\ref{thm1error} in the 3D case.

  In the 2D case, by employing more refined embedding inequalities \textcolor{black}{$$H^1(\Omega)\hookrightarrow L^q(\Omega),\,\,\forall\,2\leq  q <\infty,\,\,\text{in 2D case} $$} \textcolor{black}{and adjusting the two ndex values $p_0$ and $q_0$}
  \textcolor{black}{
  \begin{equation} 
\setlength{\abovedisplayskip}{3pt}
  2 < p_0 = \frac{2}{1 - \frac14 \epsilon_0} , \, \, \,  q_0 = 8 \epsilon_0^{-1} < +\infty , 
  \quad \mbox{so that}  \, \, \, \frac{1}{p_0} + \frac{1}{q_0} = \frac12,
\setlength{\belowdisplayskip}{3pt}
\end{equation} 
  }
  the dependence on the time-space step sizes can be relaxed to \textcolor{black}{$h^2\lesssim \Delta t \lesssim h^{\epsilon_0}$}.
\end{proof}

\begin{remark}
If the numerical design is based on the equivalent formulation~\eqref{eq1.6} of the LLG equation, the convergence analysis and error estimate will be more straightforward. In this formulation, the first nonlinear term on the right hand side is easier to handle, while the second nonlinear term does not contain a nonlinear product with $\Delta \mathbf{m}$, and only $| \nabla \mathbf{m}|$ is involved in this nonlinear expansion. In turn, in the standard $L^\infty(0,\,T;\,L^2(\Omega)) \cap L^2(0,\,T;H^1(\Omega))$ convergence analysis, the left hand side diffusion part gives a perfect $L^2(0,\,T;H^1(\Omega))$ error dissipation, and the numerical error inner product associated with the nonlinear part $| \nabla \mathbf{m}|^2 \mathbf{m}$ would only contain the terms in the form of either $\langle \mathbf{e}^{n-1}, \widetilde{\mathbf{e}}^n \rangle$ or $\langle \nabla_h \mathbf{e}^{n-1}, \widetilde{\mathbf{e}}^n \rangle$ (combined with certain nonlinear coefficient functions), and no pair of error gradient inner product is expected. These error estimates could be accomplished with the help of discrete H\"older inequality and Sobolev embedding; see the related works in  \cite{an2021optimal, gui2022convergence}, etc. 
\end{remark} 

\subsection{Further insights}
If the numerical scheme is designed to be an approximation to the formulation~\eqref{eq1.6}, a theoretical justification of the global-in-time energy dissipation would become a very challenging issue, based on the fact that, an inner product with $-\Delta \mathbf{m}$ by the right hand side of \eqref{eq1.6} would not directly lead to a valuable estimate. Because of this subtle fact, most projection-style numerical schemes based on~\eqref{eq1.6} could hardly preserve an energy dissipation at a theoretical level. 

On the other hand, if the original PDE formulation~\eqref{e_original model} is taken, the numerical design of energy dissipative scheme becomes more straightforward, due to the gradient structure on the right hand side. However, a theoretical convergence analysis and error estimate for such numerical schemes would be much more challenging, because of the numerical error associated with the nonlinear part $\mathbf{m} \times ( \mathbf{m} \times \Delta \mathbf{m})$. A careful application of summation by parts formula reveals that, the numerical error inner product associated with this nonlinear part would contain an error gradient pair inner product, namely in the form of $\langle \nabla_h \mathbf{e}^{n-1}, \nabla_h \widetilde{\mathbf{e}}^n \rangle$, combined with certain nonlinear coefficient functions. Consequently, an optimal rate error estimate could hardly go through without the error gradient coefficient control.

The convergence analysis for the proposed numerical scheme~\eqref{discretescheme}, as well as the rewritten form~\eqref{discrete}, faces the same theoretical challenges, since it is based on the original PDE formulation~\eqref{e_original model}. In the detailed expansion~\eqref{discreteip} for the weak form, the nonlinear error inner product associated with the last two terms on the right hand side would contain an error gradient pair inner product. On the other hand, the point-wise length preservation of both the numerical and exact solutions has greatly facilitated the convergence analysis. Thanks to the point-wise orthogonality identities~\eqref{lem4eq1},~\eqref{lem4eq2}, which come from the length preservation of the numerical solution, the numerical system could be transformed into an equivalent weak formulation~\eqref{m_hconsistency}. Of course, in such a reformulation, the associated nonlinear error estimate would still contain certain error gradient pair inner product terms. However, these terms become much easier to control than the ones without the reformulation. In the nonlinear error expansions, as presented in \eqref{I_1}-\eqref{I9}, it is observed that $I_1$ and $I_8$ contain certain error gradient pair inner product terms. Meanwhile, the $I_1$ term vanishes, due to the point-wise orthogonality of the two inner product functions, while the nonlinear coefficient of the $I_8$ term is of $O (\Delta t)$, which comes from the regularity of the exact solution. These state-of-arts efforts have overcome the theoretical challenges associated with the PDE formulation \eqref{e_original model}, and an optimal rate convergence analysis finally goes through. 

In fact, the PDE formulations \eqref{eq1.6} and \eqref{e_original model} are equivalent at the continuous level. In terms of the proposed numerical scheme~\eqref{discretescheme}, it is observed that the third term on the right hand side nonlinear expansions corresponds to exactly a numerical approximation to the inner product with the nonlinear term $| \nabla \mathbf{m}|^2 \mathbf{m}$ appearing in \eqref{eq1.6}. In other words, although the numerical derivation of \eqref{discretescheme} is based on the original PDE formulation \eqref{e_original model}, the magic point of the numerical reformulation~\eqref{m_hconsistency} is to transform the numerical system into a consistent approximation to an alternate PDE formulation \eqref{eq1.6}, in a weak form, combined with a few $O (\Delta t)$ numerical correction terms. As a result, the convergence analysis and error estimate could follow the standard process for \eqref{eq1.6}. With the help of a point-wise length preservation, such a magic reformulation makes thee proposed numerical schemes~\eqref{discretescheme} enjoy the theoretical convenience in both the energy dissipation analysis and error estimate.

\begin{remark} 
Other than the standard semi-implicit projection (SIP) method~\eqref{semidiscrete}, there have been some works of quasi-projection method~\cite{ChenJ2021a, xie20a}, in which the discrete temporal differentiation $\widetilde{\mathbf{m}}^{n+1} - \mathbf{m}^{n}$ is replaced by $\widetilde{\mathbf{m}}^{n+1} -\widetilde{\mathbf{m}}^{n}$. In this approach, the temporal update is only applied to the intermediate numerical variable, and the normalized original numerical variable $\mathbf{m}^n$ only plays a role of the nonlinear coefficient function in the algorithm. The convergence analysis for the quasi-projection method turns out to be much easier than the standard projection one, since there is no need to derive the renormalization stage error estimate (such as~\eqref{e_l^2} and \eqref{e_lh1} in Lemma~\ref{lem6}), and a direct $H^1$ error estimate would be feasible. On the other hand, a theoretical analysis of energy dissipation for the quasi-projection method would face a serious difficulty. 
\end{remark}

\begin{remark} 
Since the point-wise length preservation has played an essential role in the finite difference weak reformulation~\eqref{m_hconsistency}, the numerical design and theoretical analysis could be effectively extended to the mass-lumped finite element scheme, which is appropriate to handle a point-wise estimate. It is anticipated that the error analysis for the finite element version will differ from traditional finite element approaches, which constitutes a promising direction for future research works. 
\end{remark}

\begin{table}[htbp]
	\centering
	\caption{Error and convergence rates for the magnetization in discrete $L^2$ and $H^1$ norms  with $\Delta t = h^2$}
	\label{table1}
	\small
	\begin{tabular}{c|c|c|c|c} \hline\hline
		$N_x\times N_y$&$||\mathbf{e}^n||_{l^{\infty}(l^{2})}$ &$Rate$
		&$||\nabla_h\mathbf{e}^n||_{l^2(l^2)}$ &$Rate$  \\ \hline
		$8\times 8$ &4.47  &---       &5.54      &---          \\
		$16\times 16$ &2.16  &1.05     &1.83     &1.60         \\
		$32\times 32$ &6.24E-1  &1.79    &4.70E-1      &1.96        \\
		$64\times 64$ &1.64E-1  &1.93     &1.21E-1     &1.96        \\
		$128\times 128$ &4.16E-2  &1.98     &3.04E-2      &1.99         \\
		\hline\hline
	\end{tabular}
\end{table}
\begin{table}[htbp]
	\centering
	\caption{Error and convergence rates for the magnetization in discrete $L^2$ and $H^1$ norms  with $\Delta t = h$}
	\label{table2}
	\small
	\begin{tabular}{c|c|c|c|c} \hline\hline
		$N_x\times N_y$&$||\mathbf{e}^n||_{l^{\infty}(l^{2})}$ &$Rate$
		&$||\nabla_h\mathbf{e}^n||_{l^2(l^2)}$ &$Rate$  \\ \hline
		$8\times 8$ &1.90  &---       &1.56      &---          \\
		$16\times 16$ &1.02  &0.90     &7.98E-1     &0.97         \\
		$32\times 32$ &5.26E-1  &0.95    &3.99E-1      &1.00        \\
		$64\times 64$ &2.67E-1  &0.98     &1.99E-1     &1.01        \\
		$128\times 128$ &1.34E-1  &0.99     &9.89E-2      &1.00         \\
		\hline\hline
	\end{tabular}
\end{table}

\section{Numerical results} 
In this section, we present a few numerical experiments to validate the theoretical analysis. 
Four illustrative examples are provided to demonstrate the effectiveness and accuracy of the proposed method.

\subsection{Convergence rate} The convergence rate is tested for the LLG equation equation with an external force. The exact solution is set as 
\begin{equation*}
   \setlength{\abovedisplayskip}{4pt} 
    \left\{
    \begin{array}{l}
         m_e^x(x,y,t) = \sin(t+x)\cos(t+y),  \\
         m_e^y(x,y,t) = \cos(t+x)\cos(t+y),  \\
         m_e^z(x,y,t) = \sin(t+y).
    \end{array}  \right. 
    \setlength{\belowdisplayskip}{4pt} 
\end{equation*}
It is clear that $\mathbf{m}_e$ defined above satisfies an exact point-wise length preservation, $|\mathbf{m}| =1$. We take $\beta = 1,\,\gamma=1,\,T=1$ and $\Omega = [0,\,2\pi]^2$, and periodic boundary conditions are used. The initial spatial partition is taken as a $8\times8$ grid, and the time step size is refined as $\Delta t = 1/N_x^2 = 1/N_y^2$, to demonstrate both the temporal and spatial convergence rates. Tables \ref{table1} and \ref{table2} clearly indicate the second-order and first-order convergence rates, respectively, with  $\Delta t = h^2$ and $\Delta t = h$ in discrete $\|\cdot\|_2$ and $\|\nabla_h(\cdot)\|_2$ norms. These numerical results are in good agreement with the theoretical predictions in Theorem \ref{thm1error}.

\subsection{Energy dissipation} In this simulation, we verify that the constructed scheme \eqref{discretescheme} is unconditionally dissipative
with the original energy. We set 
\begin{equation*} 
  \setlength{\abovedisplayskip}{4pt}  
  T = 1,\,\beta = 1,\,\Delta t = 0.01, N_x = N_y = 100,\,\Omega = [0,\,2\pi]^2,
  \setlength{\belowdisplayskip}{4pt} 
\end{equation*} 
and use the following initial condition
\begin{equation*}
   \setlength{\abovedisplayskip}{4pt}  
    \left\{
    \begin{array}{l}
         m^x_e(x,y,0) = \cos(x)\cos(y)\sin(0.1),  \\
         m^y_e(x,y,0) = \cos(x)\cos(y)\cos(0.1),  \\
         m^z_e(x,y,0) = (1 - \cos^2(x)\cos^2(y))^{1/2}. 
    \end{array} \right. 
    \setlength{\belowdisplayskip}{4pt} 
\end{equation*}

The evolution of the initial energy, given by  $ \frac{1}{2}\int_\Omega|\nabla\mathbf{m}|^2d\mathbf{x}$, for different Gilbert damping parameters $\gamma = 0.1,\,0.5,\,1,\,10$, is presented in Figure \ref{fig1} (a). It is observed that the original energy decreases monotonically for all damping parameter values, indicating a dissipative behavior.

\begin{figure}[htbp]
	\centering
    \subfigure[Example 2]
    {\includegraphics[width=0.35\linewidth, trim = {0cm 0cm 0cm 0cm}, clip]{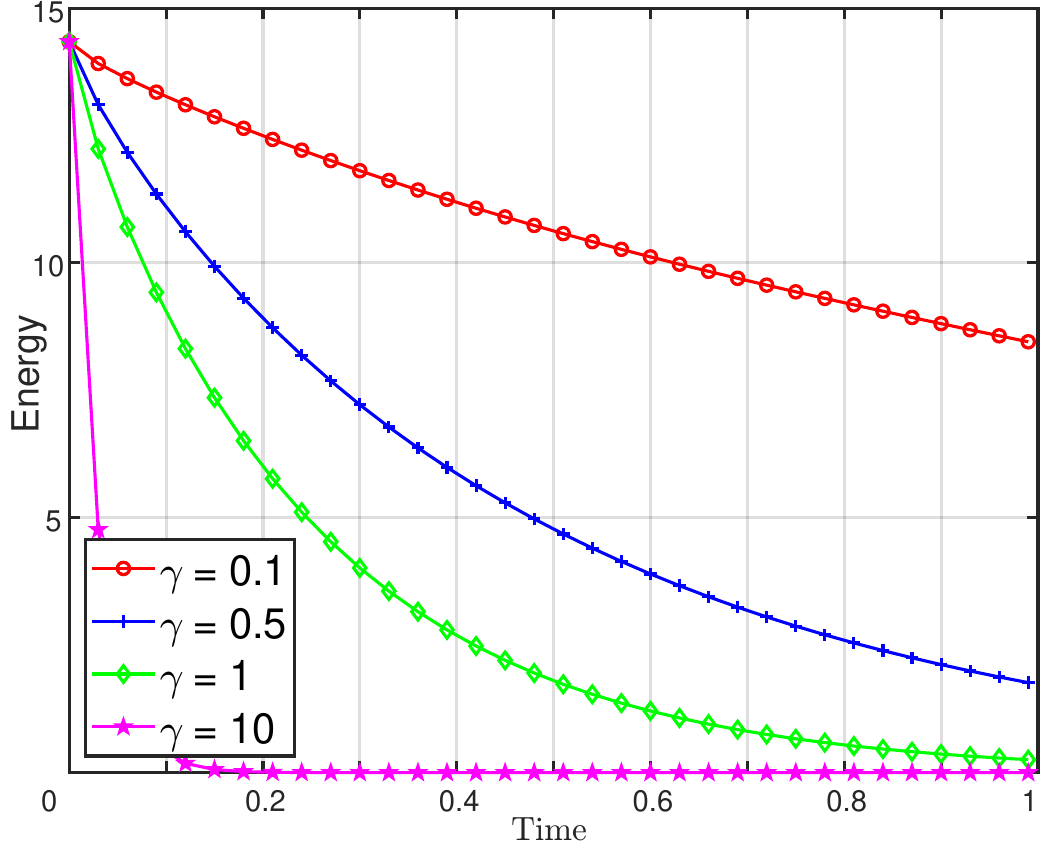}}
    \subfigure[Example 3]
    {\includegraphics[width=0.35\linewidth, trim = {0cm 0cm 0cm 0cm}, clip]{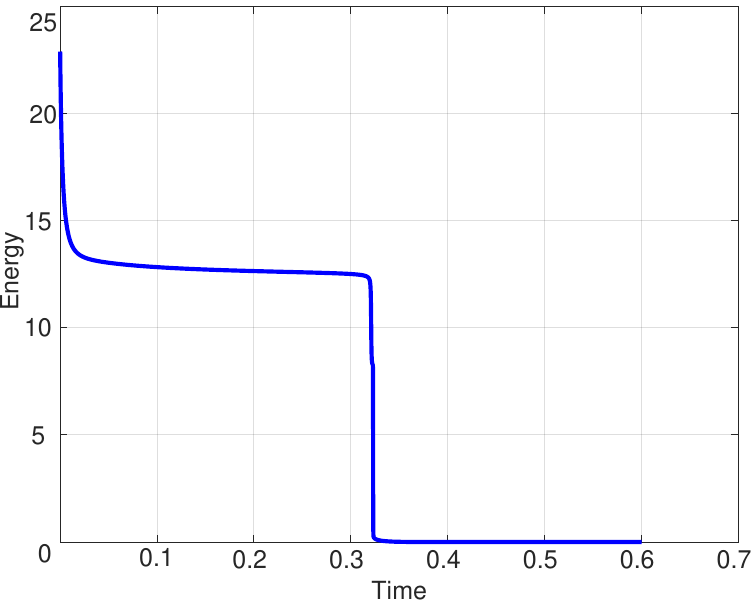}}
	\caption{Evolutions of original energy $\displaystyle \frac{1}{2}\|\nabla_h\mathbf{m}^n\|^2_{l^2}$ for Examples 2 and 3.}
	\label{fig1}
\end{figure}

\begin{figure}[htbp]
	\centering
    \subfigure[$T=0$]
	{\includegraphics[width=0.3\linewidth, trim = {0cm 0cm 0cm 0cm}, clip]{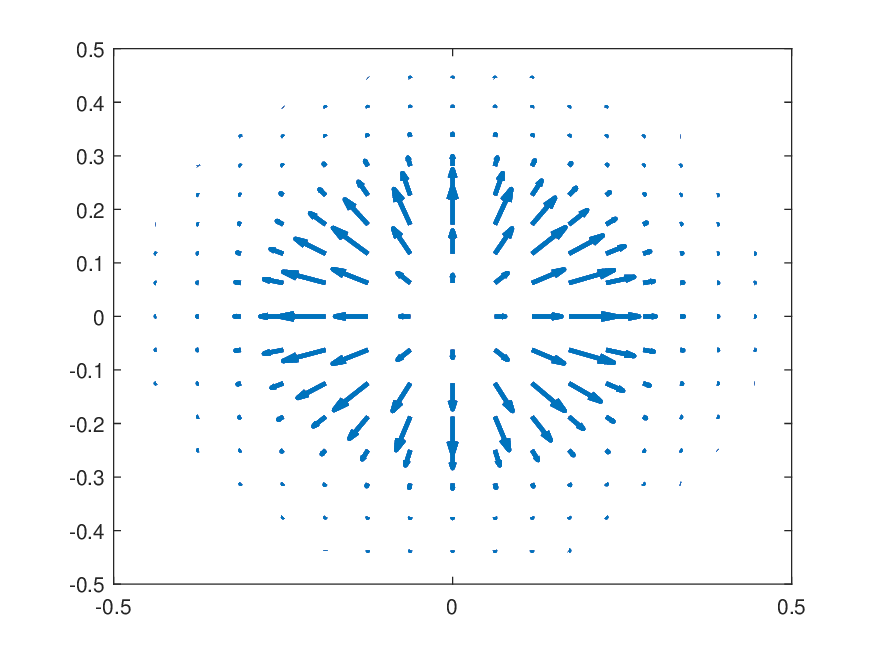}}
	\subfigure[$T=0.06$]
	{\includegraphics[width=0.3\linewidth, trim = {0cm 0cm 0cm 0cm}, clip]{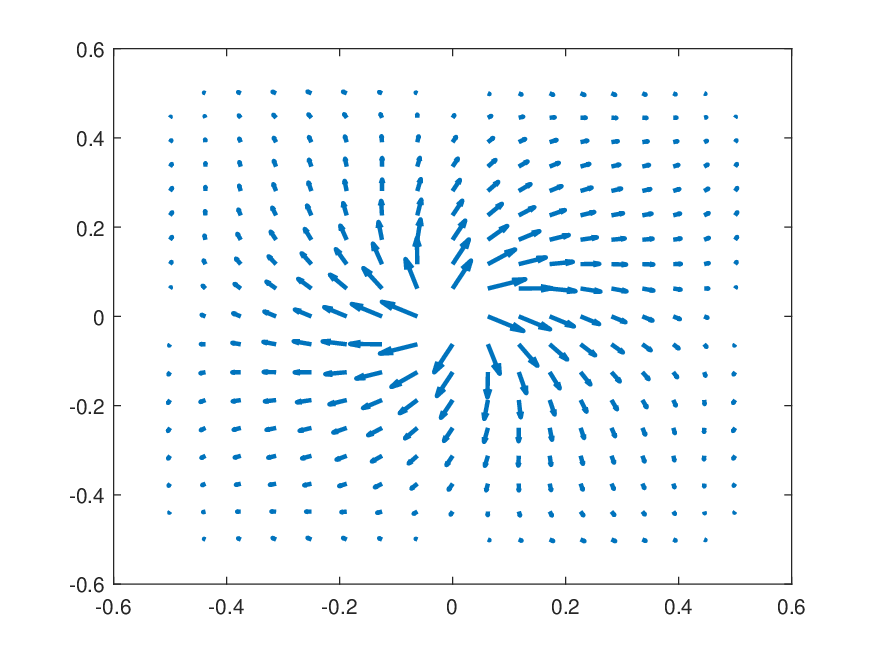}}
	\subfigure[$T=0.15$]
	{\includegraphics[width=0.3\linewidth, trim = {0cm 0cm 0cm 0cm}, clip]{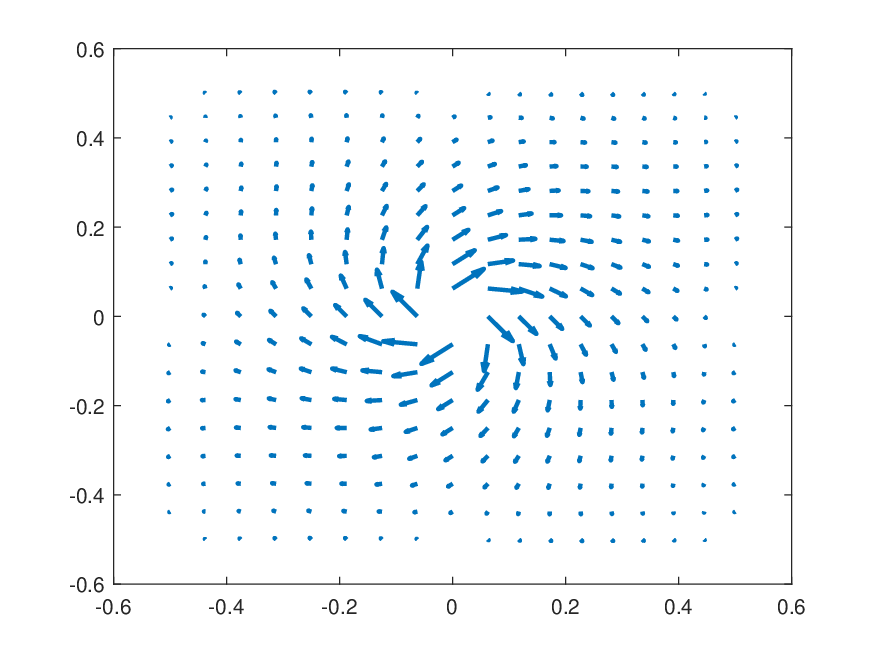}}\\
	\subfigure[$T=0.30$]
	{\includegraphics[width=0.3\linewidth, trim = {0cm 0cm 0cm 0cm}, clip]{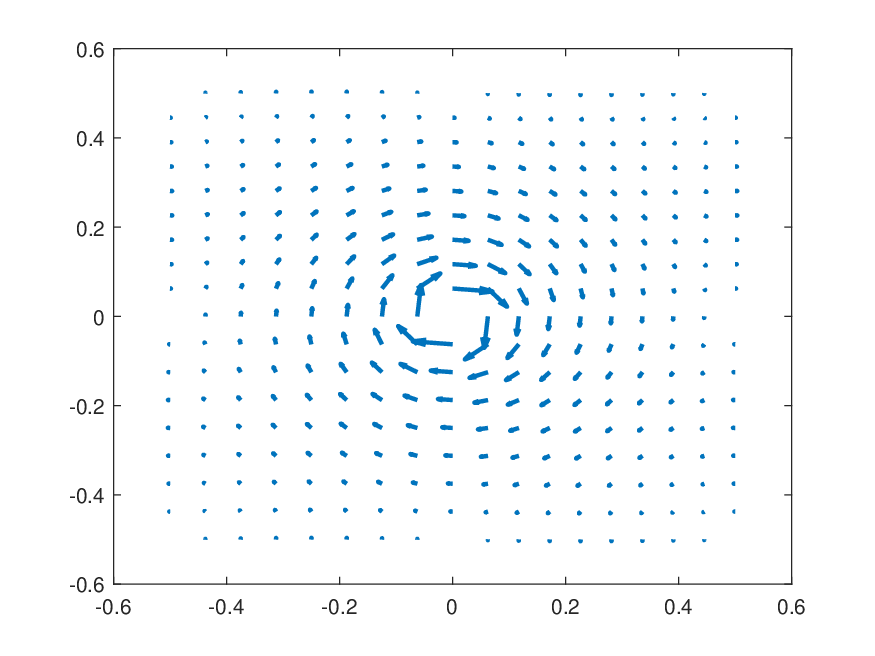}}
	\subfigure[$T=0.32$]
	{\includegraphics[width=0.3\linewidth, trim = {0cm 0cm 0cm 0cm}, clip]{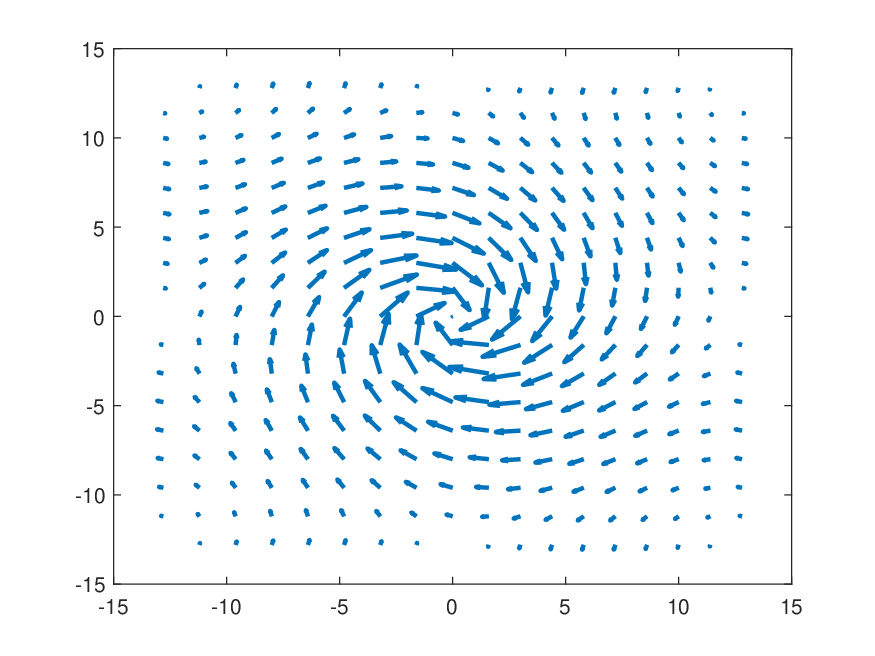}}
    \subfigure[$T=0.35$]
	{\includegraphics[width=0.3\linewidth, trim = {0cm 0cm 0cm 0cm}, clip]{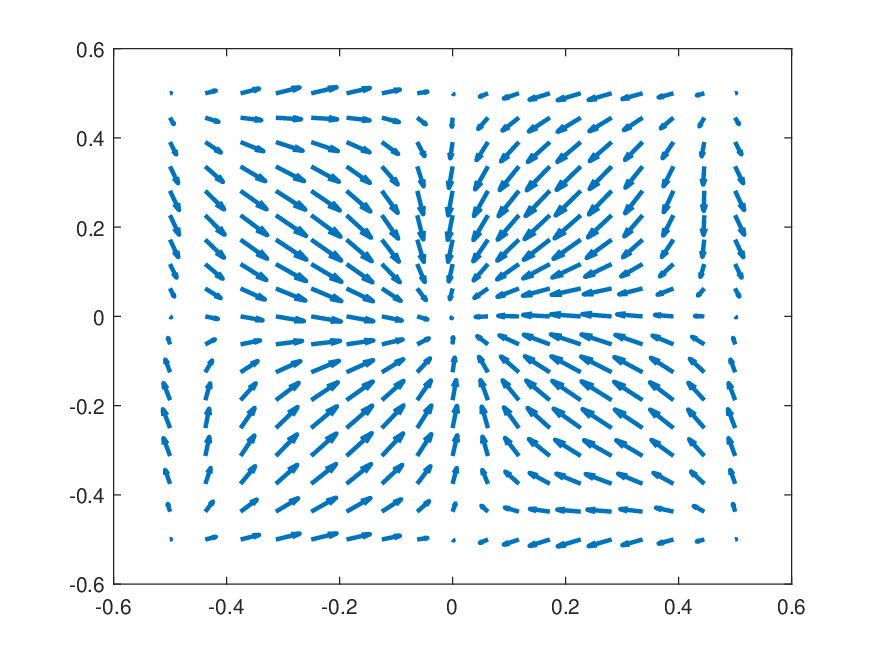}}
	\caption{Numerical magnetization $\mathbf{m}^n$ (projected on $x_1x_2$-plane) for Example 3.}
	\label{fig2}
\end{figure}

\begin{figure}[htbp]
	\centering
    \subfigure[$T=0$]
	{\includegraphics[width=0.3\linewidth, trim = {0cm 0cm 0cm 0cm}, clip]{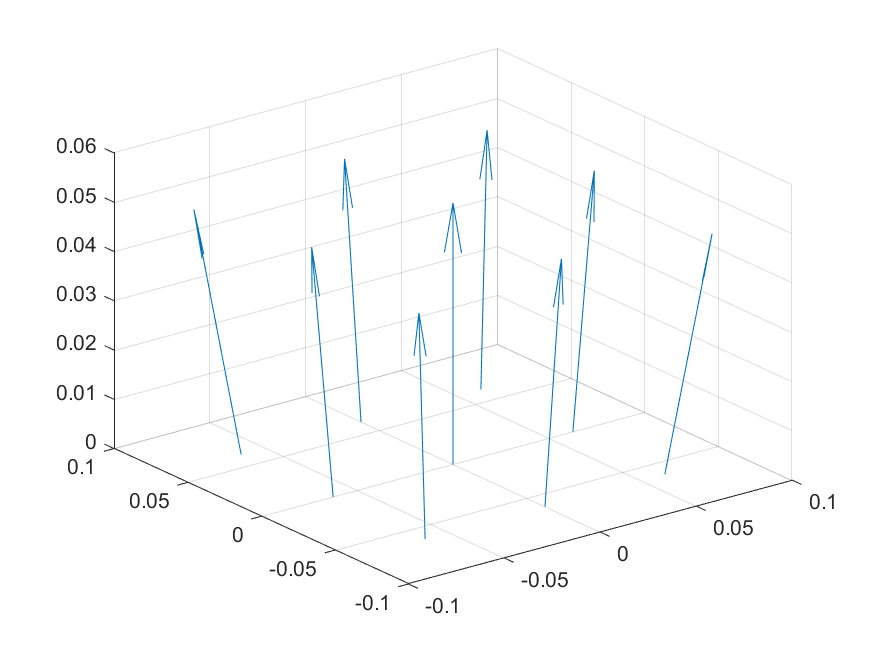}}
	\subfigure[$T=0.06$]
	{\includegraphics[width=0.3\linewidth, trim = {0cm 0cm 0cm 0cm}, clip]{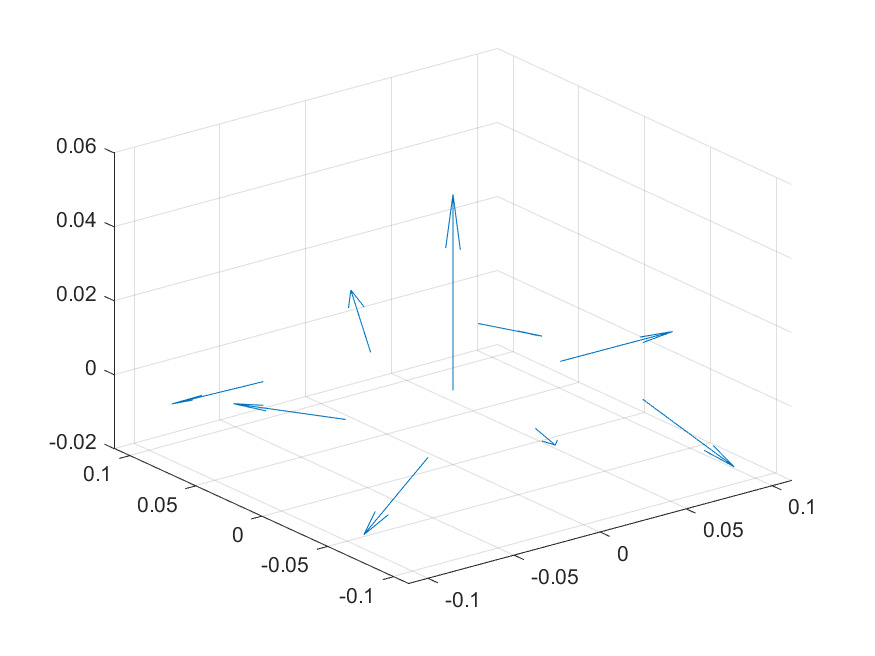}}
	\subfigure[$T=0.15$]
	{\includegraphics[width=0.3\linewidth, trim = {0cm 0cm 0cm 0cm}, clip]{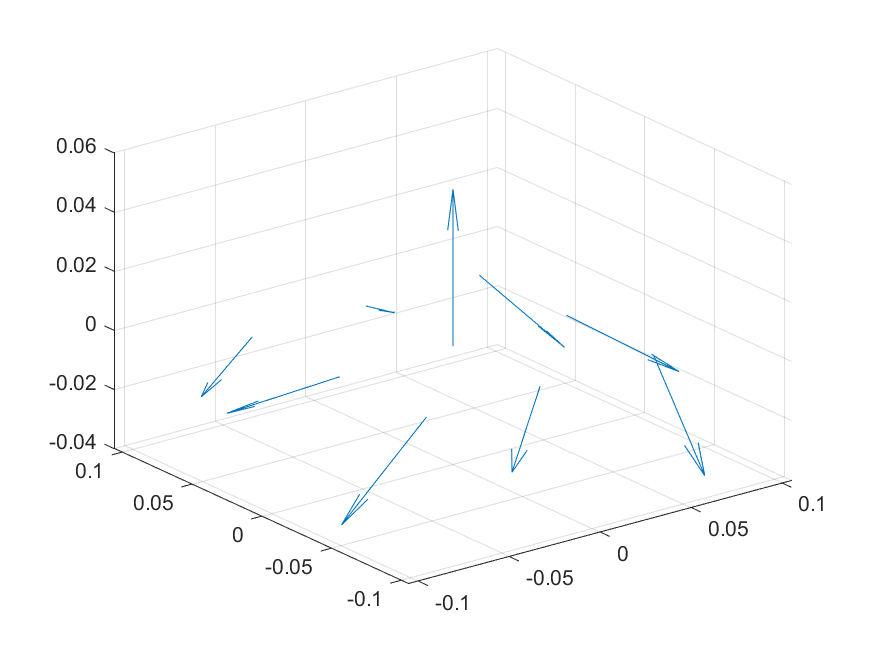}}\\
	\subfigure[$T=0.30$]
	{\includegraphics[width=0.3\linewidth, trim = {0cm 0cm 0cm 0cm}, clip]{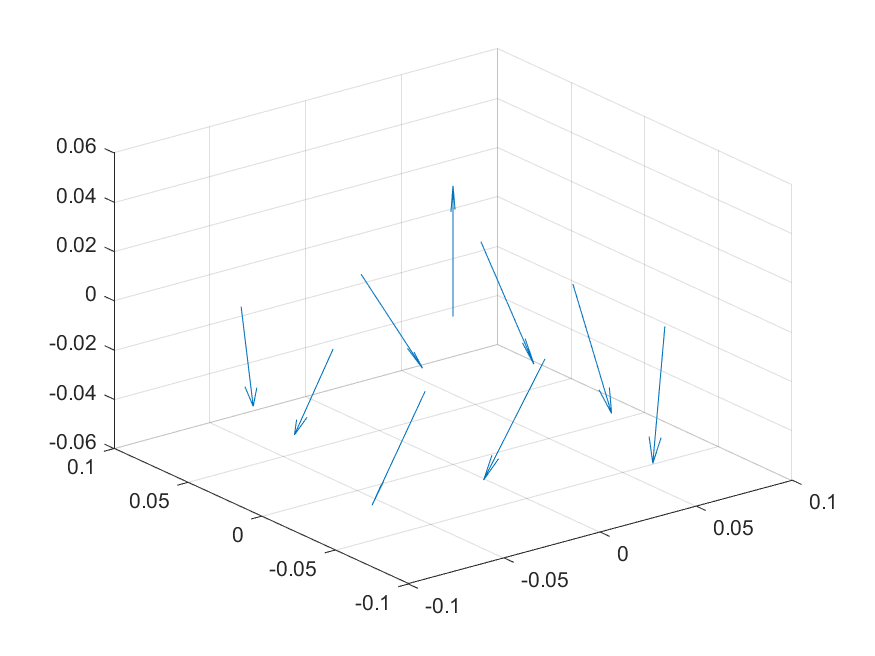}}
	\subfigure[$T=0.32$]
	{\includegraphics[width=0.3\linewidth, trim = {0cm 0cm 0cm 0cm}, clip]{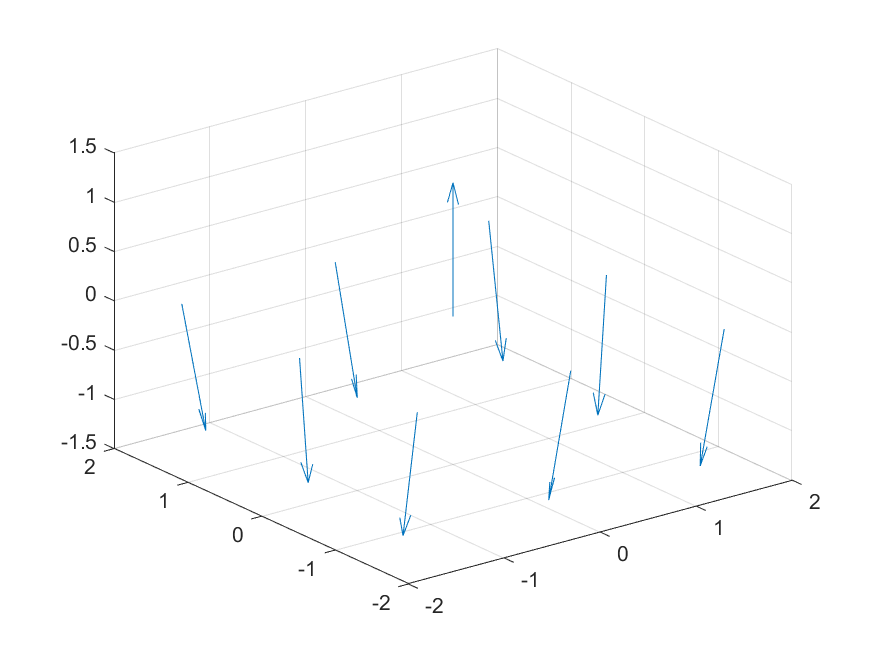}}
    \subfigure[$T=0.35$]
	{\includegraphics[width=0.3\linewidth, trim = {0cm 0cm 0cm 0cm}, clip]{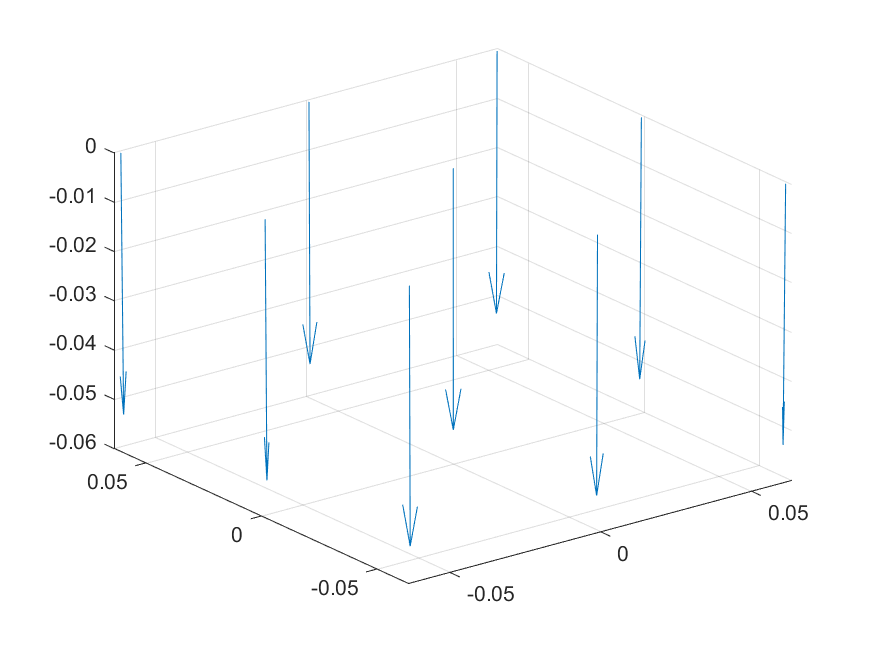}}
	\caption{Numerical magnetization $\mathbf{m}^n$ around the origin for Example 3.}
	\label{fig3}
\end{figure}

\subsection{Phenomenon of blowup}
In this subsection, we investigate the possible blowup of the LLG equation with certain smooth initial data, as given by \cite{an2021optimal, bartels2008numerical, chen1998evolution}. Set $\Omega = [-1/2,\,1/2]^2$, and let the initial data $\mathbf{m}^0$ be defined by
\begin{equation*}
   \setlength{\abovedisplayskip}{4pt}  
    \mathbf{m}^0(\mathbf{x}) = \left\{
    \begin{array}{ll}
         (0,\,0,\,-1) , &\,\forall\,|\mathbf{x}|\geq1/2,  \\
        \displaystyle (\frac{2x_1A}{A^2+|\mathbf{x}|^2},\,\frac{2x_2A}{A^2+|\mathbf{x}|^2},\,\frac{A^2 - |\mathbf{x}|^2}{A^2+|\mathbf{x}|^2}) , &\,\forall\,|\mathbf{x}|\leq1/2, 
    \end{array} \right. 
   \setlength{\belowdisplayskip}{4pt} 
\end{equation*}
with $A = (1 - 2|\mathbf{x}|)^4$. We take the the parameters $\beta = 1,\,\gamma = 1$ in the LLG equation \eqref{e_original model}.

The LLG equation is solved by the proposed scheme \eqref{discretescheme} on a uniform mesh with $h = \frac{1}{256}$ and $\Delta t = 10^{-4}$. The orthogonal projection of the vector field $\mathbf{m}^{n+1}$ on the $x_1x_2$-plane, and close-up pictures of $\mathbf{m}^{n+1}$ near the origin at a sequence of time instants, $t = 0,\,0.06, 0.15,\,0.30,\,0.32,\,0.35$, are displayed in Figures. \ref{fig2}-\ref{fig3}. It is observed that $\mathbf{m}^{n+1}$ preserves $(0,\,0,\,1)^T$ at the origin and gradually turns down to $(0,\,0,\,-1)^T$ near the origin, which is consistent with the blowup phenomenon presented in \cite{an2021optimal, bartels2008numerical}. Furthermore, the original energy is verified to be dissipative in Figure. \ref{fig1} (b). 

\subsection{Static skyrmions}
In this subsection, we investigate static skyrmions in Dzyaloshinskii-Moriya materials with easy-axis anisotropy. Magnetic skyrmions are a class of vortex-like spin configurations observed in specific magnetic materials. Due to their topological stability, nanoscale dimensions, facile manipulability, and particle-like properties, skyrmions are regarded as promising candidates for a range of advanced applications, including next-generation information storage devices and neuromorphic computing systems \cite{jia2025electrically, jiang2015blowing, nagaosa2013topological, romming2013writing}. In the previous numerical studies, a shooting method \cite{bogdanov1999stability} and a relaxation algorithm \cite{komineas2015skyrmion} have been the primary approaches to investigate skyrmion behaviors. Among these schemes, the shooting method is highly sensitive to initial conditions, whereas the relaxation algorithm transforms the original static problem into the task of determining the steady-state solution of a dissipative system. Consequently, the relaxation algorithm necessitates a strict dissipativity of the numerical scheme. Our proposed numerical method, which guarantees an unconditional original energy dissipation, effectively ensures this requirement. 

\begin{figure}[htbp]
	\centering
	\subfigure[$(m_1,\,m_2,\,m_3)$]
	{\includegraphics[width=0.35\linewidth, trim = {0cm 2.8cm 0cm 1.8cm}, clip]{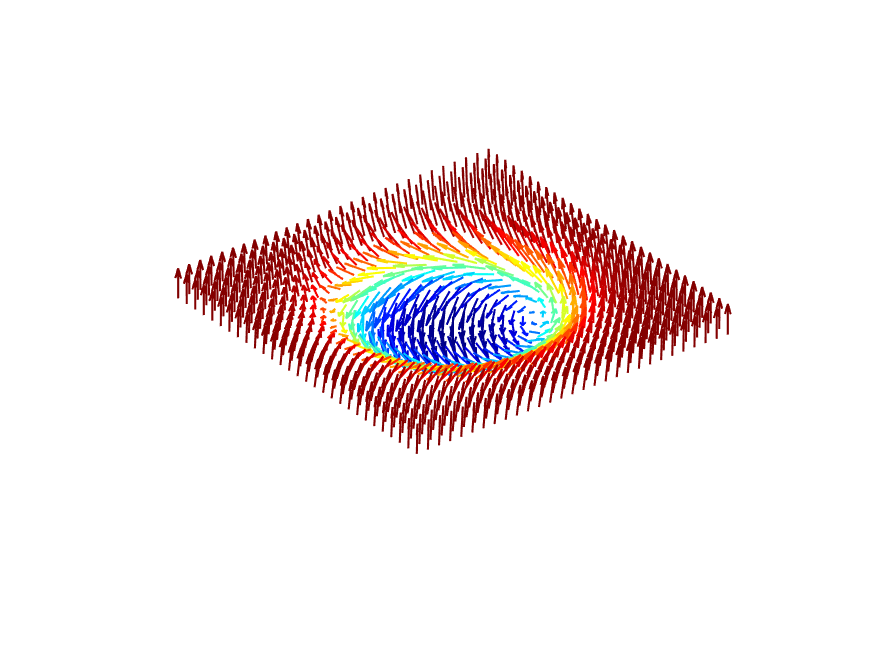}}
    \subfigure[$(m_1,\,m_2,\,m_3)$]
	{\includegraphics[width=0.35\linewidth, trim = {0cm 2.8cm 0cm 1.8cm}, clip]{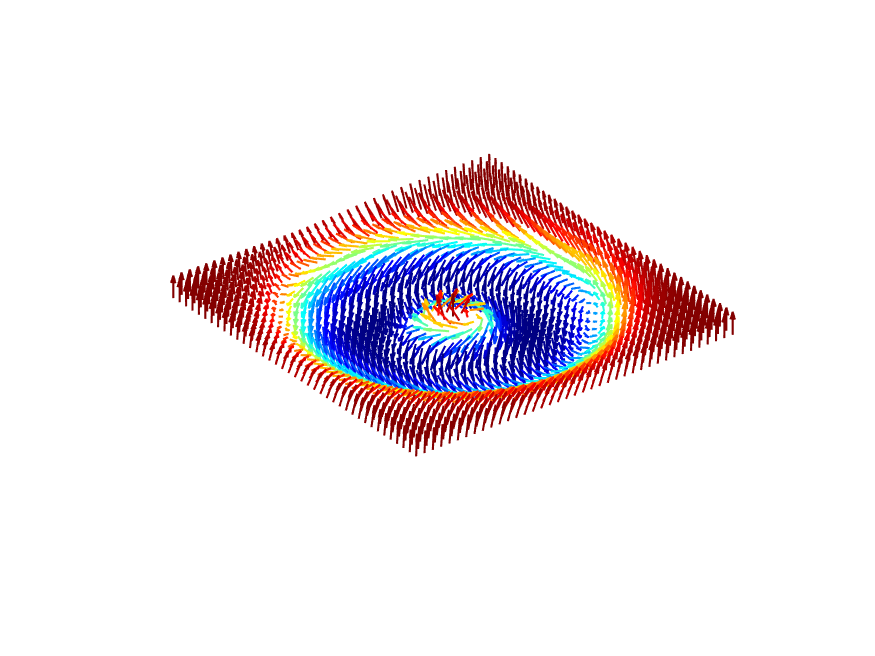}}\\
     \subfigure
	{\includegraphics[width=0.35\linewidth, trim = {2cm 3cm 2cm 1cm}, clip]{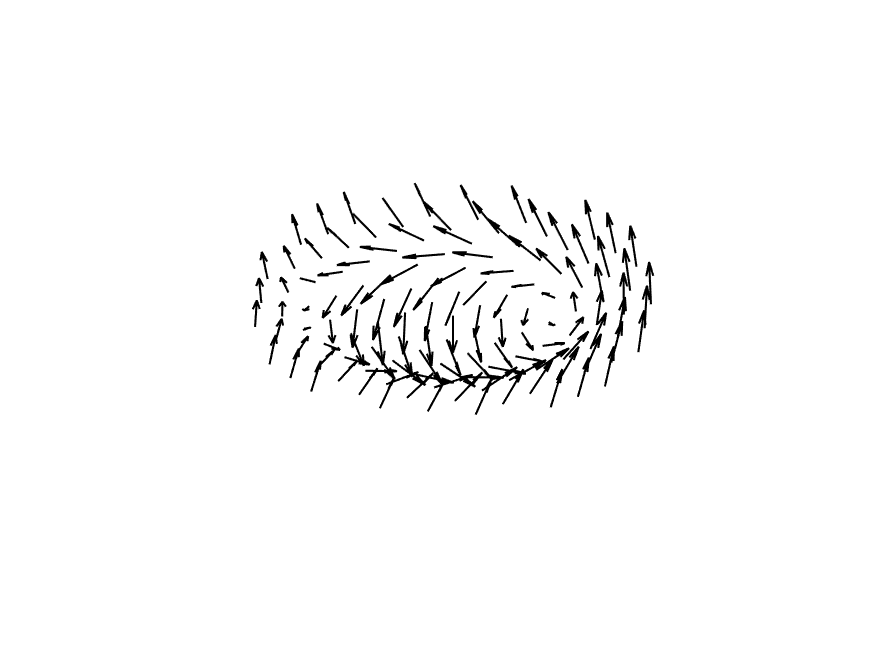}}
    \subfigure
	{\includegraphics[width=0.35\linewidth, trim = {2cm 3cm 2cm 1cm}, clip]{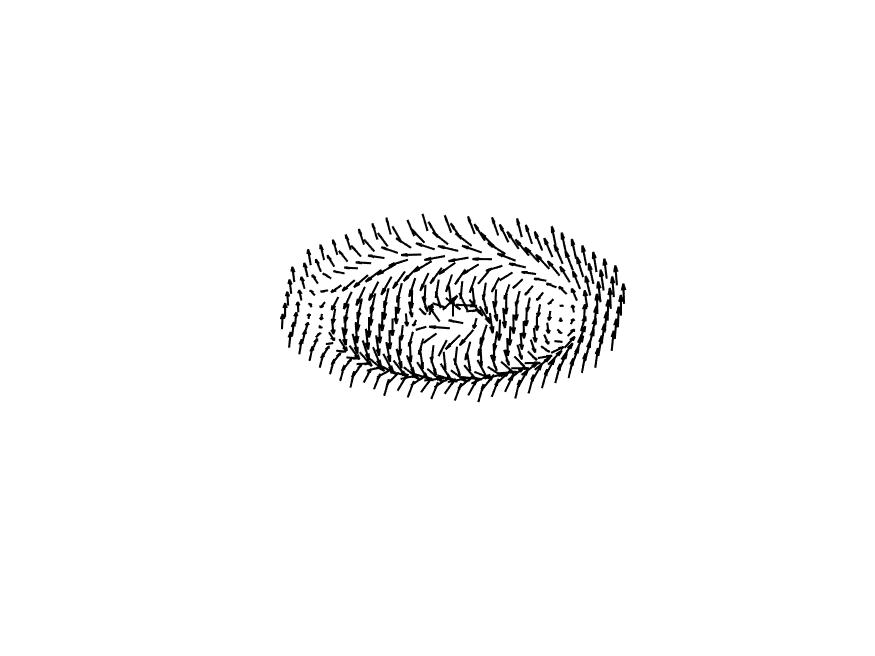}}\\
    \addtocounter{subfigure}{-2}
    \subfigure[$Q = 1$]
	{\includegraphics[width=0.35\linewidth, trim = {1cm 4.5cm 1cm 4.5cm}, clip]{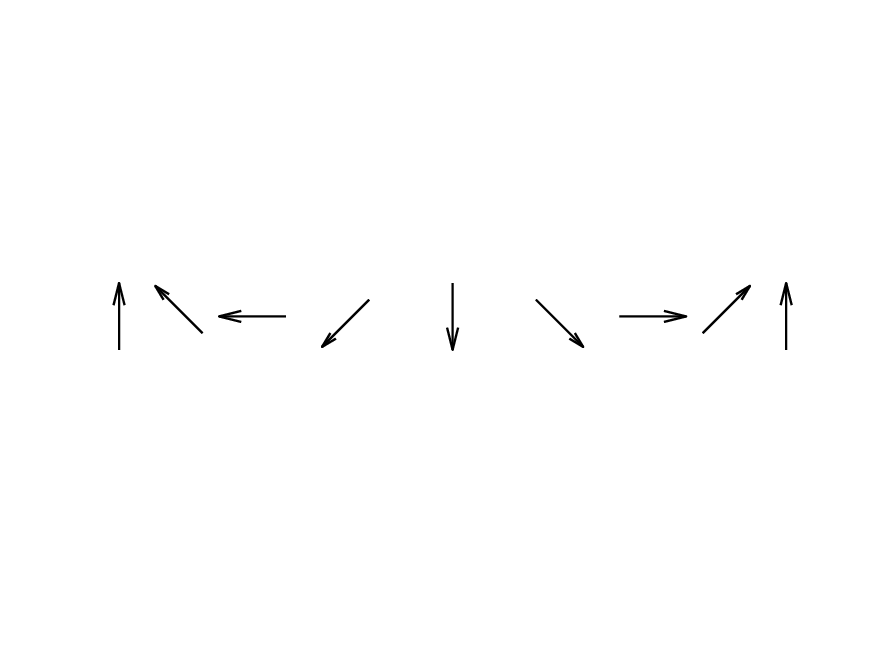}}
    \subfigure[$Q = 0$]
	{\includegraphics[width=0.35\linewidth, trim = {1cm 4.5cm 1cm 4.5cm}, clip]{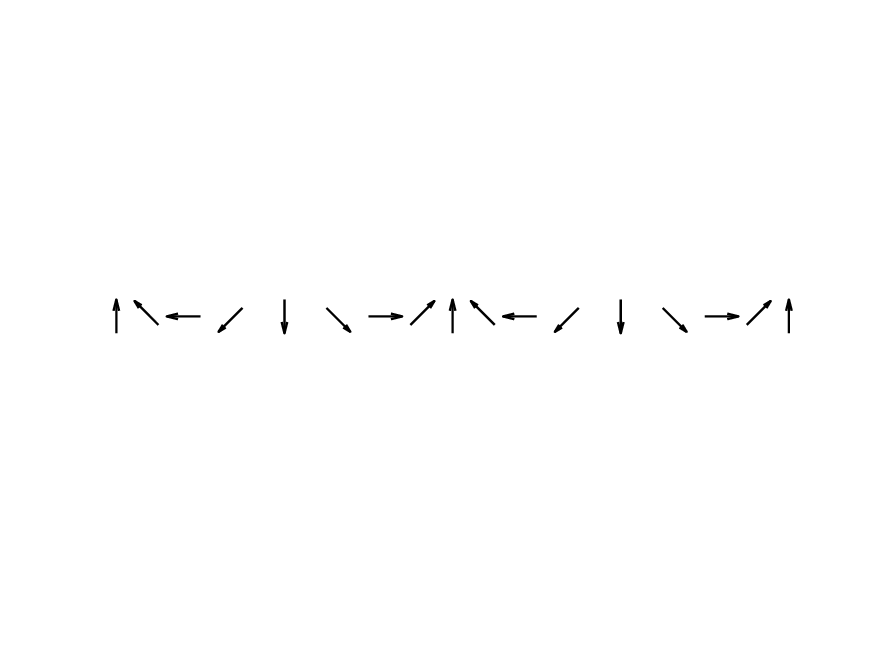}}\\
    \subfigure[$(m_1,\,m_2)$]
	{\includegraphics[width=0.35\linewidth, trim = {0cm 0cm 0cm 0cm}, clip]{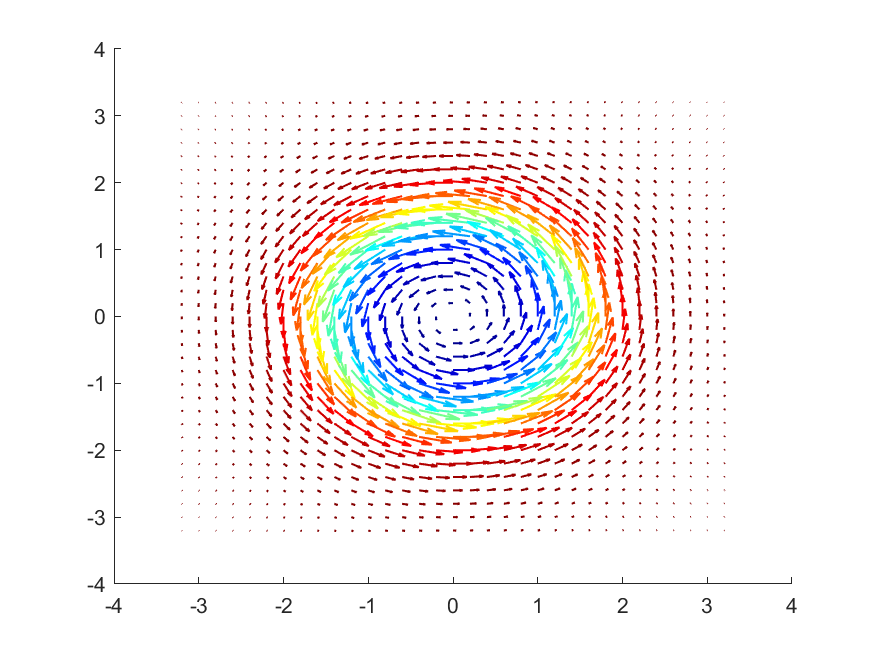}}
    \subfigure[$(m_1,\,m_2)$]
	{\includegraphics[width=0.35\linewidth, trim = {0cm 0cm 0cm 0cm}, clip]{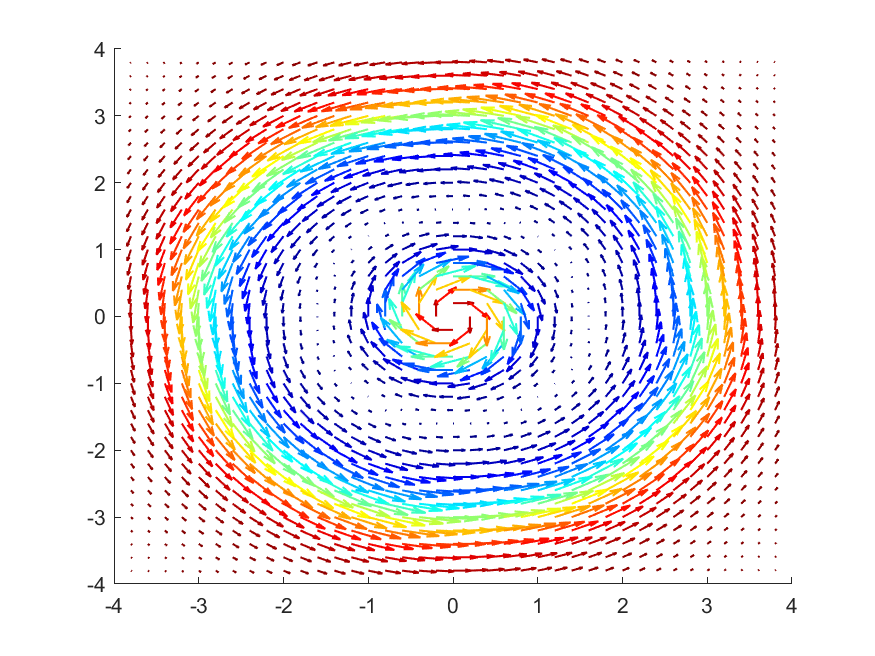}}
	\caption{The static axially symmetric skyrmion for Example 4.}
	\label{fig6}
\end{figure}
It is assumed that a thin film with easy-axis anisotropy is perpendicular to the $xy$-plane of the film and with a DM interaction energy term \cite{bogdanov1994thermodynamically}, where $\Phi(\mathbf{m}) = m_1^2 + m_2^2$. The energy functional reads
\begin{equation*}
    \setlength{\abovedisplayskip}{4pt}  
    E = \frac{1}{2} \int|\nabla \mathbf{m}|^2\,dxdy + \frac{\kappa}{2} \int \Phi(\mathbf{m})\,dxdy + \lambda \int [(\nabla \times \mathbf{m}) \cdot \mathbf{m}]\,dxdy,
   \setlength{\belowdisplayskip}{4pt} 
\end{equation*}
where $\lambda = \frac{\mathcal{D}}{|\mathcal{D}|} = \pm 1$ will be referred to as the chirality and $\kappa = \frac{K_u}{K_0}$ is the rationalized (dimensionless) anisotropy constant. For simplicity, we take $\kappa = 3,\,\lambda =1$, and see that 
\begin{equation*}
  \setlength{\abovedisplayskip}{4pt} 
    \mathbf{H_{eff}} = - \frac{\delta E}{\delta \textbf{m}} = \Delta \textbf{m} + 3 m_3\hat{\mathbf{e}}_3 - 2[\partial_2m_3 \hat{\mathbf{e}}_1- \partial_1m_3\hat{\mathbf{e}}_2 + (\partial_1m_2 - \partial_2m_1)\hat{\mathbf{e}}_3]. 
    \setlength{\belowdisplayskip}{4pt} 
\end{equation*}

It should be specially noticed that the characteristic of the magnetic configuration is defined by the skyrmion number, whose specific expression is given by  
\begin{equation*}
  \setlength{\abovedisplayskip}{4pt}  
    Q = \frac{1}{4\pi}\int \frac{1}{2}\epsilon_{ij}\textbf{m}\cdot(\partial_i\textbf{m}\times \partial_j\textbf{m})dxdy, 
  \setlength{\belowdisplayskip}{4pt} 
\end{equation*}
where $\epsilon_{ij}$ is the Levi-Civita symbol. In the polar coordinate $(\rho,\,\phi)$, the representation becomes 
\begin{equation*} 
   \setlength{\abovedisplayskip}{4pt}  
    m_1 = \sin\Theta\cos\Phi, \quad m_2=\sin\Theta\sin\Phi \quad ,m_3=\cos\Theta, 
   \setlength{\belowdisplayskip}{4pt} 
\end{equation*}
with $\Theta = \theta(\rho),\,\,\Phi = \phi + \pi/2$, and the expression for $Q$ could be transformed into
\begin{equation*} 
  \setlength{\abovedisplayskip}{4pt}  
    Q = \frac{1}{2}\int^\infty_0 \frac{dm_3}{d\rho}d\rho = \frac{1}{2}[m_3(\infty) - m_3(0)]. 
  \setlength{\belowdisplayskip}{4pt} 
\end{equation*}

It is shown in \cite{komineas2015skyrmion} that one could initialize \eqref{e_original model1} $(\beta = 0,\,\gamma = 1)$ with an essentially arbitrary spin configuration (where \(Q = 1\) with $m_3(\infty) = 1,\,m_3(0) = -1$), and this configuration will subsequently converge to a local minimum of the energy functional with a static \(Q = 1\) skyrmion. The principle of this method lies in a numerical algorithm with dissipative properties. We solve the LLG equation \eqref{e_original model1} $(\beta = 0,\,\gamma = 1)$ by the scheme \eqref{discretescheme} on a \(256 \times 256\) uniform mesh with \(h = 0.1\), \(\Delta t = 0.01\), and homogeneous Neumann boundary conditions. The numerical lattice is sufficiently large, and it is assumed that an infinite thin film is simulated in all the presented computations. The initial spin configuration will evolve under the action of \eqref{e_original model1} $(\beta = 0,\,\gamma = 1)$, causing its energy to decrease monotonically and eventually converge to the static solution of \eqref{e_original model1} $(\beta = 1,\,\gamma = 0)$. 

As for the case of \(Q = 0\), a suitable trial solution is constructed to serve as the initial condition in \eqref{e_original model1} $(\beta = 0,\,\gamma = 1)$. We review the axisymmetric \(Q = 1\) configuration constructed above, now denoted by \(\mathbf{n} = (n_1, n_2, n_3)\). Afterward, a transformation \cite{komineas1998vortex} is applied:
\begin{equation*} 
  \setlength{\abovedisplayskip}{4pt} 
  m_1 = 2n_3n_1,\,m_2 = 2n_3n_2,\,m_3=2n_3^2 - 1. 
  \setlength{\belowdisplayskip}{4pt}  
\end{equation*} 

Te static skyrmions, with \(Q = 1\) and \(Q = 0\), are presented from a three-dimensional perspective in Figures \ref{fig6}(a) and \ref{fig6}(b), respectively. The local state structures with \(Q = 1\) and \(Q = 0\) are displayed in Figures. \ref{fig6}(c) and \ref{fig6}(d), respectively. From Figures. \ref{fig6}(e) and \ref{fig6}(f), it shows the projection of the \(Q = 1\) and \(Q = 0\) skyrmion onto the \(x_1x_2-\)plane. 
All results are consistent with the computational findings presented in \cite{komineas2015skyrmion} and with the experimental images corresponding to \(\phi = \pi/2\) in \cite{gungordu2016stability}.

\section{Conclusion}
In this paper, we present and analyze a linear and fully discrete finite difference scheme to the Landau–Lifshitz–Gilbert (LLG) equation. This numerical scheme is straightforward to implement, and capable of simultaneously preserving the non-convex constraint \(|\mathbf{m}| = 1\), as well as ensuring an unconditional original energy dissipation of the original system. Moreover, through a rigorous error analysis, it is proved that the proposed scheme is able to achieve an optimal rate error estimate in the \(L^\infty(0,\,T;\,L^2(\Omega))\) and \(L^2(0,\,T;H^1(\Omega))\) norms. Such a theoretical justification of convergence analysis and error estimate turns out to be highly challenging, due to the highly complicated nonlinear structures in the numerical design. An equivalent weak form of the original numerical scheme, which comes from a point-wise length preservation of the numerical solution, has played an essential role in the theoretical derivation, since the difficulty associated with the error estimate of a nonlinear Laplacian term has been overcome. In addition, motivated by the finite difference weak form, we plan to extend the proposed numerical design and theoretical analysis to the mass-lumped finite element method in the future works. It is anticipated that the error analysis for the finite element version will differ from traditional finite element approaches, which constitutes a promising direction for future research works.
\bibliographystyle{siamplain}
\bibliography{ref}

\begin{thebibliography}{10}

\bibitem{akrivis2021higher}
{\sc G.~Akrivis, M.~Feischl, B.~Kov{\'a}cs, and C.~Lubich}, {\em Higher-order linearly implicit full discretization of the {L}andau--{L}ifshitz--{G}ilbert equation}, Mathematics of Computation, 90 (2021), pp.~995--1038.

\bibitem{alouges2008new}
{\sc F.~Alouges}, {\em A new finite element scheme for {L}andau-{L}ifchitz equations}, Discrete Contin. Dyn. Syst. Ser. S, 1 (2008), pp.~187--196.

\bibitem{alouges2006convergence}
{\sc F.~Alouges and P.~Jaisson}, {\em Convergence of a finite element discretization for the {L}andau--{L}ifshitz equations in micromagnetism}, Mathematical Models and Methods in Applied Sciences, 16 (2006), pp.~299--316.

\bibitem{an2021optimal}
{\sc R.~An, H.~Gao, and W.~Sun}, {\em Optimal error analysis of {E}uler and {C}rank--{N}icolson projection finite difference schemes for {L}andau--{L}ifshitz equation}, SIAM Journal on Numerical Analysis, 59 (2021), pp.~1639--1662.

\bibitem{an2025optimal}
{\sc R.~An, Y.~Li, and W.~Sun}, {\em Optimal error analysis of the normalized tangent plane {FEM} for {L}andau--{L}ifshitz--{G}ilbert equation}, IMA Journal of Numerical Analysis, 45 (2025), pp.~3109--3137.

\bibitem{an2022analysis}
{\sc R.~An and W.~Sun}, {\em Analysis of backward euler projection {FEM} for the {L}andau--{L}ifshitz equation}, IMA Journal of Numerical Analysis, 42 (2022), pp.~2336--2360.

\bibitem{badia2011finite}
{\sc S.~Badia, F.~Guill{\'e}n-Gonz{\'a}lez, and J.~V. Guti{\'e}rrez-Santacreu}, {\em Finite element approximation of nematic liquid crystal flows using a saddle-point structure}, Journal of Computational Physics, 230 (2011), pp.~1686--1706.

\bibitem{bao2013optimal}
{\sc W.~Bao and Y.~Cai}, {\em Optimal error estimates of finite difference methods for the {G}ross-{P}itaevskii equation with angular momentum rotation}, Mathematics of Computation, 82 (2013), pp.~99--128.

\bibitem{bartels2008numerical}
{\sc S.~Bartels, J.~Ko, and A.~Prohl}, {\em Numerical analysis of an explicit approximation scheme for the {L}andau-{L}ifshitz-{G}ilbert equation}, Mathematics of Computation, 77 (2008), pp.~773--788.

\bibitem{bogdanov1994thermodynamically}
{\sc A.~Bogdanov and A.~Hubert}, {\em Thermodynamically stable magnetic vortex states in magnetic crystals}, Journal of magnetism and magnetic materials, 138 (1994), pp.~255--269.

\bibitem{bogdanov1999stability}
{\sc A.~Bogdanov and A.~Hubert}, {\em The stability of vortex-like structures in uniaxial ferromagnets}, Journal of magnetism and magnetic materials, 195 (1999), pp.~182--192.

\bibitem{cai2022second}
{\sc Y.~Cai, J.~Chen, C.~Wang, and C.~Xie}, {\em A second-order numerical method for {L}andau-{L}ifshitz-{G}ilbert equation with large damping parameters}, Journal of Computational Physics, 451 (2022), p.~110831.

\bibitem{Cai2023a}
{\sc Y.~Cai, J.~Chen, C.~Wang, and C.~Xie}, {\em Error analysis of a linear numerical scheme for the {Landau-Lifshitz} equation with large damping parameters}, Mathematical Methods in the Applied Sciences, 46 (2023), pp.~18592--18974.

\bibitem{ChenJ2021a}
{\sc J.~Chen, C.~Wang, and C.~Xie}, {\em Convergence analysis of a second-order semi-implicit projection method for {Landau-Lifshiz} equation}, Applied Numerical Mathematics, 168 (2021), pp.~55--74.

\bibitem{chen20b}
{\sc W.~Chen, C.~Wang, S.~Wang, X.~Wang, and S.~Wise}, {\em Energy stable numerical schemes for ternary {Cahn-Hilliard} system}, Journal of Scientific Computing, 84 (2020), p.~27.

\bibitem{chen1998evolution}
{\sc Y.~Chen and F.~H. Lin}, {\em Evolution equations with a free boundary condition}, The Journal of Geometric Analysis, 8 (1998), pp.~179--197.

\bibitem{cheng2023length}
{\sc Q.~Cheng and J.~Shen}, {\em Length preserving numerical schemes for {L}andau--{L}ifshitz equation based on {L}agrange multiplier approaches}, SIAM Journal on Scientific Computing, 45 (2023), pp.~A530--A553.

\bibitem{cohen1989relaxation}
{\sc R.~Cohen, S.-Y. Lin, and M.~Luskin}, {\em Relaxation and gradient methods for molecular orientation in liquid crystals}, Computer Physics Communications, 53 (1989), pp.~455--465.

\bibitem{du2025semi}
{\sc Q.~Du, S.~Liu, and J.~Yang}, {\em Semi-implicit projection schemes for manifold constraint gradient flows}, Journal of Scientific Computing, 103 (2025), p.~92.

\bibitem{weinan2001numerical}
{\sc W.~E and X.-P. Wang}, {\em Numerical methods for the {L}andau-{L}ifshitz equation}, SIAM journal on numerical analysis,  (2001), pp.~1647--1665.

\bibitem{gui2022convergence}
{\sc X.~Gui, B.~Li, and J.~Wang}, {\em Convergence of renormalized finite element methods for heat flow of harmonic maps}, SIAM Journal on Numerical Analysis, 60 (2022), pp.~312--338.

\bibitem{gungordu2016stability}
{\sc U.~G{\"u}ng{\"o}rd{\"u}, R.~Nepal, O.~A. Tretiakov, K.~Belashchenko, and A.~A. Kovalev}, {\em Stability of skyrmion lattices and symmetries of quasi-two-dimensional chiral magnets}, Physical Review B, 93 (2016), p.~064428.

\bibitem{guo1993landau}
{\sc B.~Guo and M.-C. Hong}, {\em The {L}andau-{L}ifshitz equation of the ferromagnetic spin chain and harmonic maps}, Calculus of Variations and Partial Differential Equations, 1 (1993), pp.~311--334.

\bibitem{HeSu07}
{\sc Y.~He and W.~Sun}, {\em Stability and convergence of the {C}rank--{N}icolson/{A}dams--{B}ashforth scheme for the time-dependent {N}avier--{S}tokes equations}, SIAM Journal on Numerical Analysis, 45 (2007), pp.~837--869.

\bibitem{jia2025electrically}
{\sc J.~Jia, J.~Ren, S.~Zhou, Z.~Zeng, H.~Lin, Y.~Hu, Z.~Li, Y.~Shen, Z.~Chen, X.~Chen, et~al.}, {\em Electrically tuning photonic topological quasiparticles in synthetic two-level system}, Nature Physics,  (2025), pp.~1--8.

\bibitem{jiang2015blowing}
{\sc W.~Jiang, P.~Upadhyaya, W.~Zhang, G.~Yu, M.~B. Jungfleisch, F.~Y. Fradin, J.~E. Pearson, Y.~Tserkovnyak, K.~L. Wang, O.~Heinonen, et~al.}, {\em Blowing magnetic skyrmion bubbles}, Science, 349 (2015), pp.~283--286.

\bibitem{kim2017mimetic}
{\sc E.~Kim and K.~Lipnikov}, {\em The mimetic finite difference method for the {L}andau--{L}ifshitz equation}, Journal of Computational Physics, 328 (2017), pp.~109--130.

\bibitem{komineas1998vortex}
{\sc S.~Komineas and N.~Papanicolaou}, {\em Vortex dynamics in two-dimensional antiferromagnets}, Nonlinearity, 11 (1998), p.~265.

\bibitem{komineas2015skyrmion}
{\sc S.~Komineas and N.~Papanicolaou}, {\em Skyrmion dynamics in chiral ferromagnets}, Physical Review B, 92 (2015), p.~064412.

\bibitem{landau1992theory}
{\sc L.~Landau and E.~Lifshitz}, {\em On the theory of the dispersion of magnetic permeability in ferromagnetic bodies}, in Perspectives in Theoretical Physics, Elsevier, 1992, pp.~51--65.

\bibitem{li2026class}
{\sc X.~Li, J.~Shen, and N.~Zheng}, {\em On a class of higher-order length preserving and energy decreasing {IMEX} schemes for the {L}andau-{L}ifshitz equation}, Journal of Computational Physics,  (2026), p.~114786.

\bibitem{li2005numerical}
{\sc Z.~Li, L.~Vulkov, and J.~W{\'a}sniewski}, {\em Numerical {A}nalysis and Its {A}pplications: Third International Conference}, vol.~3401, Springer, 2005.

\bibitem{liu2000approximation}
{\sc C.~Liu and N.~J. Walkington}, {\em Approximation of liquid crystal flows}, SIAM Journal on Numerical Analysis, 37 (2000), pp.~725--741.

\bibitem{nagaosa2013topological}
{\sc N.~Nagaosa and Y.~Tokura}, {\em Topological properties and dynamics of magnetic skyrmions}, Nature nanotechnology, 8 (2013), pp.~899--911.

\bibitem{prohl2001computational}
{\sc A.~Prohl et~al.}, {\em Computational micromagnetism}, Springer, 2001.

\bibitem{romming2013writing}
{\sc N.~Romming, C.~Hanneken, M.~Menzel, J.~E. Bickel, B.~Wolter, K.~Von~Bergmann, A.~Kubetzka, and R.~Wiesendanger}, {\em Writing and deleting single magnetic skyrmions}, Science, 341 (2013), pp.~636--639.

\bibitem{samarskii1976difference}
{\sc A.~Samarskii and V.~Andreev}, {\em Difference methods for elliptic equations(russian book)}, Moscow, Izdatel'stvo Nauka, 1976. 352,  (1976).

\bibitem{shen1990long}
{\sc J.~Shen}, {\em Long time stability and convergence for fully discrete nonlinear {G}alerkin methods}, Applicable Analysis, 38 (1990), pp.~201--229.

\bibitem{suess2002time}
{\sc D.~Suess, V.~Tsiantos, T.~Schrefl, J.~Fidler, W.~Scholz, H.~Forster, R.~Dittrich, and J.~J. Miles}, {\em Time resolved micromagnetics using a preconditioned time integration method}, Journal of Magnetism and Magnetic Materials, 248 (2002), pp.~298--311.

\bibitem{xie20a}
{\sc C.~Xie, J.~Garcia-Cervera, C.~Wang, Z.~Zhou, and J.~Chen}, {\em Second-order semi-implicit projection methods for micromagnetics simulations}, Journal of Compututational Physics, 404 (2020), p.~109104.

\bibitem{yang1998dynamical}
{\sc B.~Yang and D.~R. Fredkin}, {\em Dynamical micromagnetics by the finite element method}, IEEE transactions on magnetics, 34 (1998), pp.~3842--3852.

\bibitem{yu1995general}
{\sc Z.~Yu-lin}, {\em General interpolation formulas for spaces of discrete functions with nonuniform meshes}, Journal of Computational Mathematics,  (1995), pp.~70--92.

\end{thebibliography}
\end{document}